[1994/12/01]
\documentclass[mn,fleqn]{w-art}
\usepackage{times}
\usepackage{amsmath}
\usepackage{w-thm}
\usepackage{bm}
\usepackage{enumerate}
\usepackage{amssymb,latexsym,amsmath,array}
\usepackage{makeidx,amsfonts}
\usepackage{longtable,graphicx,ifthen,calc}

\numberwithin{equation}{section}
\chardef\bslash=`\\ 

\newcommand{\eval}[2][\right]{\relax
\ifx#1\right\relax \left.\fi#2#1\rvert}

\def\qed{\hfill $\square$}
\def\squar{\vbox{\hrule\hbox{\vrule height 6pt \hskip
6pt\vrule}\hrule}}

\def\qed{\hfill $\squar$}
\def\squar{\vbox{\hrule\hbox{\vrule height 6pt \hskip
6pt\vrule}\hrule}}

\begin{document}

\

\

{\it Submitted to Mathematische Nachrichten, 2006

Registered as Technical Report 2006-1 at the University of Siena.}

\

\

\Volume{} \Year{} 
\pagespan{1}{}
\Receiveddate{} \Reviseddate{} \Accepteddate{} \Dateposted{}





\title[Short Title]{A continuation principle for a class of periodically
perturbed autonomous systems}


\author[Sh. First Author]{Mikhail Kamenskii \inst{1}}\address[\inst{1}]{Department of Mathematics,
Voronezh State University, Voronezh, Russia}

\author[Sh. Second Author]{Oleg Makarenkov\inst{1}}

\author[Sh. Third Author]{Paolo Nistri \footnote{Corresponding author:
    e-mail: {\sf pnistri@dii.unisi.it}, Phone: +39\,0577233603, Fax:
    +39\,0577233602}\inst{2}}\address[\inst{2}]{Dipartimento di Ingegneria dell' Informazione,
    Universit\`a di Siena, 53100 Siena, Italy}

\thanks{Supported by the research project GNAMPA
:``Qualitative Analysis and Control of Hybrid Systems'', by RFBR
grants 02-01-00189, 05-01-00100, by U.S.CRDF - RF Ministry of
Education grant VZ-010 and by President of Russian Federation
Fellowship for Scientific Training Abroad}

\begin{abstract}
The paper deals with a $T$-periodically perturbed autonomous
system in $\mathbb{R}^n$ of the form
\begin{equation}\tag{PS}
\dot x=\psi(x)+\varepsilon\phi(t,x,\varepsilon)
\end{equation}
with $\varepsilon>0$ small. The main goal of the paper is to
provide conditions ensuring the existence of $T$-periodic
solutions to (PS) belonging to a given open set $W\subset
C([0,T],\mathbb{R}^n).$ This problem is considered in the case
when the boundary $\partial W$ of $W$ contains at most a finite
number of nondegenerate $T$-periodic solutions of the autonomous
system $\dot x=\psi(x)$. The starting point of our approach is the
following property due to Malkin: if for any $T$-periodic limit
cycle $x_0$ of $\dot x=\psi(x)$ belonging to $\partial W$ the
so-called bifurcation function $f_{x_0}(\theta),$
$\theta\in[0,T],$ associated to $x_0$, see (\ref{biffun}),
satisfies the condition $f_{x_0}(0)\not=0$ then the integral
operator
\[
(Q_\varepsilon x)(t)=x(T)+\int_0^t
\psi(x(\tau))d\tau+\varepsilon\int_0^t
\phi(\tau,x(\tau),\varepsilon)d\tau, \quad t\in[0,T],
\]
does not have fixed points on $\partial W$ for all $\varepsilon>0$
sufficiently small. By means of the Malkin's bifurcation function
we then establish a formula to evaluate the Leray-Schauder
topological degree of $I-Q_\varepsilon$ on $W.$ This formula
permits to state existence results that generalize or improve
several results of the existing literature. In particular, we
extend a continuation principle due to Capietto, Mawhin and
Zanolin where it is assumed that $\partial W$ does not contain any
$T$-periodic solutions of the unperturbed system. Moreover, we
obtain generalizations or improvements of some existence results
due to Malkin and Loud.
\end{abstract}

\maketitle

\section { Introduction}

The aim of this paper is to provide conditions ensuring the
existence of $T$-periodic solutions to the $T$-periodically
perturbed system of the form
\begin{equation}\label{ps}
  \dot x=\psi(x)+\varepsilon\phi(t,x,\varepsilon)
\end{equation}
belonging to a given set $W\subset C([0,T],\mathbb{R}^n).$ Here we
assume that
\begin{equation}\label{reg}
\psi\in C^1(\mathbb{R}^n,\mathbb{R}^n){\rm\ and\ }
\phi:\mathbb{R}\times\mathbb{R}^n\times[0,1]\to\mathbb{R}^n{\rm\
satisfies\ Carath\acute{e}odory\ type\ conditions,}\end{equation}
i.e. $\phi(\cdot,x,\varepsilon)$ is (Lebesgue) measurable for each
$(x,\varepsilon),$ $\phi(t,\cdot,\cdot)$ is continuous for almost
all (a.a.) $t$ and, for each $r>0$ there exists $b_r\in
L^1([0,T],\mathbb{R}_+)$ such that $|\phi(t,x,\varepsilon)|\le
b_r(t)$ for a.a. $t\in[0,T]$ and all $|x|\le r,$
$\varepsilon\in[0,1].$ Moreover, $\phi$ is $T$-periodic in time
and any solution $x\in C([0,T],\mathbb{R}^n)$ to (\ref{ps})
satisfying the boundary condition
\begin{equation}\label{kam}
  x(0)=x(T)
\end{equation}
will be called a $T$-periodic solution to (\ref{ps}). Any
$T$-periodic function $x\in C([0,T],\mathbb{R}^n)$ is considered
as extended from $[0,T]$ to $\mathbb{R}$ by $T$-periodicity.
Moreover, any constant function $x\in C([0,T],\mathbb{R}^n)$ is
identified with the vector $x(0)$ of $\mathbb{R}^n.$ Let
$Q_\varepsilon:C([0,T],\mathbb{R}^n)\to C([0,T],\mathbb{R}^n)$ be
the integral operator given by
\[
  (Q_\varepsilon x)(t)=x(T)+\int_0^t \psi(x(\tau))d\tau+\varepsilon\int_0^t
  \phi(\tau,x(\tau),\varepsilon)d\tau,\quad t\in[0,T],\ \varepsilon>0,
\]
whose fixed points are $T$-periodic solutions to (\ref{ps}). In
the case when
\begin{equation}\label{mawc1}
Q_0 x\not=x\quad {\rm for\ any\ } x\in\partial W
\end{equation}
and
\begin{equation}\label{mawc2}
  d_{\mathbb{R}^n}(\psi,W\cap\mathbb{R}^n)\not=0,
\end{equation}
where $d_{\mathbb{R}^n}(\psi,W\cap\mathbb{R}^n)$ is the Brouwer
topological degree of $\psi$ in $W\cap\mathbb{R}^n,$ the existence
problem of $T$-periodic solutions to (\ref{ps}) has been solved by
Capietto, Mawhin and Zanolin in \cite{maw}. In fact, they proved
(\cite{maw}, Corollary 1), that under conditions (\ref{mawc1}) and
(\ref{mawc2}) the following formula holds
\begin{equation}\label{mform}
  d(I-Q_0,W)=(-1)^n d_{\mathbb{R}^n}(\psi,W\cap\mathbb{R}^n),
\end{equation}
where $d(I-Q_0,W)$ is the Leray-Schauder topological degree of
$I-Q_0$ in $W.$  It follows from (\ref{mform}) that
\begin{equation}\label{mform1}
  d(I-Q_\varepsilon,W)=(-1)^n d_{\mathbb{R}^n}(\psi,W\cap\mathbb{R}^n)
\end{equation}
for any $\varepsilon>0$ sufficiently small. Therefore under
conditions (\ref{mawc1}) and (\ref{mawc2}) system (\ref{ps}) has a
$T$-periodic solution in $W$ for any perturbation term $\phi$ and
any sufficiently small $\varepsilon>0.$  Observe that the
assumption (\ref{mawc2}) implies that the set $W$ contains a
constant solution of
\begin{equation}\label{np}
  \dot x=\psi(x).
\end{equation}
In \cite{maw} many relevant examples satisfying conditions
(\ref{mawc1}) and (\ref{mawc2}) are provided. Moreover, the
authors have focused several results due to I.~Berstein and
A.~Halanay, J.~Cronin, A.~Lando, E.~Muhamadiev and others, which
have been generalized or improved.

The main goal of this paper is to provide conditions on the
perturbation term $\phi$ in such a way that, for $\varepsilon>0$
sufficiently small, $d(I-Q_\varepsilon,W)$ is defined and
different from zero for a wider class of sets $W.$ Indeed, through
the paper we will not require (\ref{mawc1}), i.e. we will allow
$\partial W$  to contain fixed point of $Q_0.$ Under this more
general condition, we will establish a formula for
$d(I-Q_\varepsilon,W)$ that guarantees, under suitable conditions
on $\phi$, that $d(I-Q_\varepsilon,W)\neq 0$ even in the case when
$d_{\mathbb{R}^n}(\psi,W\cap\mathbb{R}^n)=0.$ Precisely, we assume
that
\begin{equation}\label{c2a}
{\rm the\ set\ } \mathfrak{S}_W=\left\{x\in\partial W:Q_0
x=x\right\}{\rm\ is\ finite,}
\end{equation}
and for any $x_0\in\mathfrak{S}_W$  the linearized system
\begin{equation}\label{ls}
  \dot y=\psi'(x_0(t))y
\end{equation}
has the  characteristic multiplier $1$ of multiplicity 1, i.~e.
any $x_0\in\mathfrak{S}_W$ is a nondegenerate limit cycle of
(\ref{np}). It is clear that, under assumption (\ref{c2a}), the
topological degree $d(I-Q_0,W)$ is not necessarily defined. The
approach proposed in this paper to overcome this difficulty
consists in introducing the Malkin's bifurcation function
\begin{equation}\label{biffun}
f_{x_0}(\theta) = {\rm sign}\left<\dot
x_0(0),z_0(0)\right>\int\limits_0^T\left<z_0(\tau),
\phi(\tau-\theta,x_0(\tau),0)\right>d\tau,\end{equation} where
$z_0$ is a nontrivial $T$-periodic solution of the adjoint system
\begin{equation}\label{as}
\dot z=-(\psi'(x_0(t)))^*z.
\end{equation}
From \cite{mal} (or \cite{malb}, Theorem p.~387) we have that if
\begin{equation}\label{maln}
f_x(0)\not=0\quad{\rm for\ any\ }x\in\mathfrak{S}_W \end{equation}
then
\begin{equation}\label{mallemma}
{\rm for \ every\ }\varepsilon>0{\rm \ sufficiently\ small\ the\
topological\ degree\ } d(I-Q_\varepsilon,W) {\rm\  is\  defined.}
\end{equation}
In this paper we prove in Theorem \ref{thm2} that if (\ref{maln})
is satisfied then, for all $\varepsilon>0$ sufficiently small, we
have
\begin{equation}\label{kmnformi}
d(I-Q_\varepsilon,W)=(-1)^n
d_{\mathbb{R}^n}(\psi,W\cap\mathbb{R}^n)-\sum_{x\in\mathfrak{S}_W:~\Theta_W(x)
\not=\emptyset}(-1)^{\beta(x)}
d_{\mathbb{R}}\left(f_x,\left(0,\min\{\Theta_W(x)\}\right)\right),
\end{equation}
where
\[
\Theta_W(x)=\left\{\theta_0\in(0,T):S_{\theta_0}\,x\in\partial W,\
S_\theta\,
  x\in W {\rm\ for\ any\ }\theta\in(0,\theta_0)\right\},
  \quad{\rm for\ any\ }x\in\mathfrak{S}_W,
\]
\[
(S_\theta\,x)(t) = x(t+\theta) \quad{\rm and}
\]
\[
\beta(x_0)\; {\rm is\ the\ sum\ of\ the\ multiplicities\ of\ the\
characteristic\ multipliers\ greater \ than\ } 1 {\rm \ of\
}(\ref{ls}).
\]
Therefore it follows that for any perturbation term $\phi$
satisfying conditions (\ref{c2a}), (\ref{maln}) if
\begin{equation}\label{kmncond}
(-1)^n
d_{\mathbb{R}^n}(\psi,W\cap\mathbb{R}^n)-\sum_{x\in\mathfrak{S}_W:~\Theta_W(x)
\not=\emptyset}(-1)^{\beta(x)}{
d}_{\mathbb{R}}\left(f_x,\left(0,\min\{\Theta_W(x)\}\right)\right)\not=0
\end{equation}
then, for any $\varepsilon>0$ sufficiently small, system
(\ref{ps}) has a $T$-periodic solution in $W.$ Observe that if
(\ref{mawc2}) is not satisfied, but there exist at least one
$x\in\mathfrak{S}_W$ such that $\Theta_W(x)\not=\emptyset$ then
assumption (\ref{kmncond}) can be fulfilled by a suitable choice
of the perturbation term $\phi.$ In this sense assumption
(\ref{kmncond}) is weaker than (\ref{mawc2}).

The second term on the right hand side of (\ref{kmnformi}) is
similar to that of the Krasnosel'skii-Zabreyko's formula for
computing the index of a degenerate fixed point of $Q_0$ by means
of a reduction to a subspace (in our case one-dimensional), see
(\cite{krazab}, formula 24.13). However, the related
Krasnosel'skii-Zabreyko result (\cite{krazab}, Theorem~24.1) can
be applied only in the case when the operator $Q_0$ has a
particular form ensuring that $Q_0$ has only isolated fixed
points. This is not our case since any $T$-periodic cycle of
(\ref{np}) is a non-isolated fixed point of $Q_0.$

Furthermore, observe that the case when $\mathfrak{S}_W$ is
nonempty was already treated in the literature. For instance, if
$\psi=0$ then any solution of (\ref{np}) is $T$-periodic,
$\mathfrak{S}_W=\partial W$ and $d(I-Q_\varepsilon,W)$ can be
evaluated by means of the following formula due to Mawhin, see
(\cite{mawPhD1} and \cite{mawPhD2})
\begin{equation}\label{mawPhD}
  d(I-Q_\varepsilon,W)=d_{\mathbb{R}^n}\left(-\int\limits_0^T
  \phi(\tau,\cdot,0)d\tau,W\cap
  \mathbb{R}^n\right).
\end{equation}
Mawhin proved (\ref{mawPhD}) in the case when $\varepsilon>0$ is
not necessarily small. The same formula can be also used when
$\psi\not=0,$ but any solution of (\ref{np}) in $\overline{W}$ is
$T$-periodic (see \cite{sch}, formulas 3.1-3.3). This assumption
has been considerably weakened by the authors in \cite{dan} for a
wide class of sets $W.$ Specifically, in \cite{dan} it was assumed
that there exists $U\subset\mathbb{R}^n$ such that $W$ is the set
of all continuous functions from $[0,T]$ to $U$ and any point of
$\partial U$ is the initial condition of a $T$-periodic solution
to (\ref{np}), namely it was still assumed that $\mathfrak{S}_W$
is an infinite subset of $\partial W$. For $\varepsilon>0$
sufficiently small formula (\ref{mawPhD}) was expressed as
follows, see also (\cite{non}, formula 46),
\begin{equation}\label{mawPhD1}
  d(I-Q_\varepsilon,W)=d_{\mathbb{R}^n}\left(-\int\limits_0^T
  \left(x'_{(2)}(\tau,\cdot)\right)^{-1}\phi(\tau,x(\tau,\cdot),0)
  d\tau, U\right),
\end{equation}
where $x(\cdot,\xi)$ is the solution of (\ref{np}) satisfying
$x(0,\xi)=\xi.$ Hence, if
$d_{\mathbb{R}^n}(\psi,W\cap\mathbb{R}^n)=0$ then (\ref{kmnformi})
can be considered as a further development of (\ref{mawPhD}) for
the special case when $\mathfrak{S}_W$ is finite.
In fact, the following formula holds, (see (\ref{gf}) in the proof
of next Theorem \ref{thm1}),
\[ f_{x_0}(\theta)={\rm sign}\left<\dot x_0(0),z_0(0)\right>\left<
\int\limits_0^T
  \left(x'_{(2)}(\tau,\cdot)\right)^{-1}\phi(\tau,x(\tau,\cdot),0)
  d\tau,
z_0\left(\theta\right)\right>\quad{\rm for\ any\ }\theta\in[0,T].
\]

The paper is organized as follows. Section 2 is devoted to the
proof of formula (\ref{kmnformi}) and its variants. In section 3
by different choices of the set $W$ we obtain several new
existence results for $T$-periodic solutions to (\ref{ps}). In
particular, we generalize or improve some existence results due to
Loud and Malkin proved in \cite{loud} and \cite{mal} respectively.

\section {Main results}

Let $x^{-1}(t,\cdot)$ be the inverse of $x(t,\cdot),$ that is
$x(t,x^{-1}(t,\xi))=\xi$ for any $t\in\mathbb{R}$ and any $\xi\in
\mathbb{R}^n.$ For any set $V$ of $\mathbb{R}^n,$ define the set
$W_V$ of $C([0,T],\mathbb{R}^n)$ by
\[
W_V=\left\{\widehat {x}\in
C([0,T],\mathbb{R}^n):x^{-1}(t,\widehat{ x}(t))\in V,\ {\rm for\
any\ } t\in [0,T]\right\}.
\]
Clearly, $W_V$ is open in $C([0,T],\mathbb{R}^n)$ provided that
$V$ is open in $\mathbb{R}^n.$ In the sequel by $B_\delta(A)$ we
denote the $\delta$-neighborhood of the set $A$ with respect to
the norm of the space containing $A.$ The following result is
crucial for the proof of our Theorem  \ref{thm2}, but it has also
an independent interest for some applications as shown in Section
3.

\begin{theorem}\label{thm1}
Let $x_0$ be a nondegenerate $T$-periodic limit cycle of system
(\ref{np}). Let $0\le\theta_1<\theta_2\le \theta_1+\frac{T}{p},$
where $p\in\mathbb{N}$ and $\frac{T}{p}$ is the least period of
$x_0.$ Assume that $f_{x_0}(\theta_1)\not=0$ and
$f_{x_0}(\theta_2)\not=0.$ Then, for a given $\alpha>0,$ there
exist $\delta_0>0$ and a family of open sets
$\{V_\delta\}_{\delta\in(0,\delta_0]}$ satisfying the properties

\noindent 1) $x_0((\theta_1,\theta_2))\subset V_\delta\subset
B_{\delta}(x_0((\theta_1,\theta_2))),$

\noindent 2) $ \partial V_\delta\cap
x_0([\theta_1,\theta_2])=\{x_0(\theta_1),x_0(\theta_2)\}$

\noindent and such that for any $\delta\in(0,\delta_0]$ and any
$\varepsilon\in(0,\delta^{1+\alpha}]$ the degree
$d(I-Q_\varepsilon,W_{V_\delta})$ is defined and it can be
evaluated by the following formula
\[d(I-Q_\varepsilon,W_{V_\delta})=\,-\,(-1)^{\beta(x_0)}d_{\mathbb{R}}
(f_{x_0},(\theta_1,\theta_2)).
\]
\end{theorem}

Now we introduce some preliminary notions and results necessary
for the proof of the theorem. Let $x_0$ be a nondegenerate limit
cycle of (\ref{np}), then there exists, see e.g. (\cite{dem},
Lemma 1  Chap.~IV, \S 20), a fundamental matrix $Y(t)$ of system
(\ref{ls}) having the form
\begin{equation}\label{Yt}
  Y(t)=\Phi(t)\left(\begin{array}{cc} {\rm e}^{\Lambda t} & 0_{{n-1}\times 1} \\
  0_{1\times{n-1}}
  & 1 \end{array}\right),
\end{equation}
where $\Phi$ is a $T$-periodic Floquet matrix and $\Lambda$ is a
constant $(n-1)\times(n-1)$-matrix with eigenvalues different from
$0.$ In (\ref{Yt}) it is denoted by $0_{ i\times j}$ the $i\times
j$ zero matrix, in the sequel we will omit these subindexes when
confusion will not arise. For any $\delta>0$ define the set
$C_\delta\subset\mathbb{R}^n$ as follows
\[C_\delta=\left\{\zeta\in\mathbb{R}^n:\|P_{n-1}\zeta\|<\delta,\
\zeta^n\in\left(-\frac{\theta_2-\theta_1}{2},\frac{\theta_2-\theta_1}
{2}\right)\right\},\] where
\[P_{n-1}\zeta=\left(\begin{array}{c}\zeta^1\\ \vdots \\ \zeta^{n-1} \\
0
 \end{array}\right)
\]
$\zeta^k$ is the $k$-th component of the vector $\zeta$ and
$\theta_1$, $\theta_2$ are as in Theorem \ref{thm1}. Let
$\Gamma:B_\Delta(C_\delta)\to \Gamma(B_\Delta(C_\delta)),$ $
\Delta>0,$ be as follows
\[
 \Gamma(\zeta)= \frac{Y(\zeta^n+\overline{\theta})}{\|Y\|_{M_T}}
 P_{n-1}\zeta+x_0(\zeta^n+\overline{\theta}),
\]
where \[ \overline{\theta}=\frac{\theta_1+\theta_2}{2}\quad{\rm
and}\quad \|Y\|_{M_T}=\max_{\theta\in[0,T]}\|Y(\theta)\|.\]

We have the following preliminary properties.
\begin{lemma}\label{lemma1} $\left<Y(\theta)P_{n-1}\zeta,z_0(\theta)\right>=0$ for any
$\theta\in[0,T]$ and any $\zeta\in\mathbb{R}^n.$ Moreover, if
$\left<\xi,z_0(\theta)\right>=0$ for any $\theta\in [0,T]$, then
there exists $\zeta\in\mathbb{R}^n$ such that
$\left<Y(\theta)P_{n-1}\zeta,z_0(\theta)\right>=0$ for any
$\theta\in [0,T].$
\end{lemma}

\begin{proof} Let $\zeta\in\mathbb{R}^n$ and define
\[
  \widehat{\zeta}=\left(\begin{array}{cc}\left(I-{\rm e}^{\Lambda
  T}\right)^{-1} & 0 \\ 0 & 0 \end{array}\right)\zeta.
\]
By Perron's lemma \cite{perron} we have
\[
  \left<Y(\theta+T)P_{n-1}\widehat{\zeta},z_0(\theta)\right>=\left<
  Y(\theta)P_{n-1}\widehat{\zeta},z_0(\theta)\right>\quad
  {\rm for\ any\ }\theta\in[0,T].
\]
Therefore
\begin{equation}
\begin{aligned}
0=&
\left<\left(Y(\theta)-Y(\theta+T)\right)P_{n-1}\widehat{\zeta},z_0(\theta)\right>\nonumber\\
=& \left<\Phi(\theta)\left(\begin{array}{cc} {\rm
e}^{\Lambda\theta}\left(I-{\rm e}^{\Lambda T}\right) & 0 \\ 0 & 0
\end{array}\right)P_{n-1}\widehat{\zeta},z_0(\theta)\right>\nonumber\\
=& \left<\Phi(\theta)\left(\begin{array}{cc} {\rm
e}^{\Lambda\theta} & 0 \\ 0 & 0
\end{array}\right)P_{n-1}\zeta,z_0(\theta)\right>=\left<Y(\theta)P_{n-1}
\zeta,z_0(\theta)\right>\quad{\rm for\ any\ }\theta\in[0,T].
\end{aligned}
\end{equation}
To prove the second assertion define
$$
  L_\xi=\left\{\xi\in\mathbb{R}^n:\left<\xi,z_0(\theta)\right>=0\right\},\quad
  L_\zeta=\bigcup_{\zeta\in\mathbb{R}^n}Y(\theta)P_{n-1}\zeta.
$$
$L_{\xi}$ and $L_{\zeta}$ are linear subspaces of $\mathbb{R}^n$
and ${\rm dim}L_\xi=n-1.$ Since, for any $\theta\in [0,T]$,
$Y(\theta)P_{n-1}$ is a linear nonsingular map acting from
$P_{n-1}\mathbb{R}^n$ to $Y(\theta)P_{n-1}\mathbb{R}^n,$ then
${\rm dim}L_\zeta={\rm dim}P_{n-1}\mathbb{R}^n=n-1.$ But by the
first assertion of the lemma $L_\xi\supset L_\zeta$ and thus we
can conclude that $L_\xi=L_\zeta.$

\end{proof}

\begin{lemma}\label{lemma2} For any $\Delta\in(0,\Delta_0]$
and any $\delta\in(0,\delta_0]$ we have that $\Gamma$ is a
homeomorphism of $B_\Delta(C_\delta)$ onto
$\Gamma(B_\Delta(C_\delta))$ provided that $\Delta_0>0$ and
$\delta_0>0$ are sufficiently small. Moreover, the set
$\Gamma(B_\Delta(C_\delta))$ is open in $\mathbb{R}^n$ and
$\Gamma^{-1}$ is continuously differentiable in
$\Gamma(B_\Delta(C_\delta)).$
 \end{lemma}

\begin{proof} Obviously $\Gamma$ is continuous. Let us show
that $\Gamma:B_\Delta(C_\delta)\to \Gamma(B_\Delta(C_\delta))$ is
injective for $\Delta>0$ and $\delta>0$ sufficiently small. For
this assume the contrary, thus there exist
$\{a_k\}_{k\in\mathbb{N}},\{b_k\}_{k\in\mathbb{N}}\subset\mathbb{R}^n,$
$a_k\not= b_k,$ $a_k\to a_0,$ $b_k\to b_0$ as $k\to\infty,$
\begin{equation}\label{Pform}P_{n-1}a_0=P_{n-1}b_0=0,
\end{equation} such that
\begin{equation}\label{contr}
\frac{Y(a_k^n)}{\|Y\|_{M_T}}P_{n-1}a_k +
x_0(a_k^n)=\frac{Y(b_k^n)} {\|Y\|_{M_T}}P_{n-1}b_k + x_0(b_k^n).
\end{equation}
Without loss of generality we may assume that either
$a^n_k=b^n_k$ for any $k\in\mathbb{N}$ or $a^n_k\not=b^n_k$ for
any $k\in\mathbb{N}.$ Assume that $a^n_k=b^n_k$ for any
$k\in\mathbb{N},$ thus
\[
  Y(a_k^n)(P_{n-1}a_k-P_{n-1}b_k)=0\quad{\rm for\ any\
  }k\in\mathbb{N},
\]
and so
\[ P_{n-1}a_k=P_{n-1}b_k\quad{\rm for\ any\
  }k\in\mathbb{N}, \]
contradicting the property that $a_k\not= b_k$ for any
$k\in\mathbb{N}.$ Consider now the case when $a^n_k\not=b^n_k$ for
any $k\in\mathbb{N},$ from (\ref{contr}) we have
$x_0(a_0^n)=x_0(b_0^n).$ Moreover, since
$0\le\theta_1<\theta_2\le\theta_1+\frac{T}{p},$ by our choice of
$\theta_1$ and $\theta_2$, for $\Delta>0$ and $\delta>0$
sufficiently small we have that $|a_0^n-b_0^n|<\frac{T}{p}$, where
$\frac{T}{p}$ is the least period of $x_0,$ thus
$a_0^n=b_0^n=:\theta_0.$ By using Lemma \ref{lemma1}, from
(\ref{contr}) we have
\[
  \left<x_0(a_k^n)-x_0(b_k^n),z_0(a_k^n)\right>=\left<\frac{Y(b_k^n)}
  {\|Y\|_{M_T}}P_{n-1}b_k^n,z_0(a_k^n)\right>=
  \left<\frac{Y(b_k^n)-Y(a_k^n)}{\|Y\|_{M_T}}P_{n-1}b_k^n,z_0(a_k^n)\right>,
\]
or equivalently, by dividing by $a_k^n-b_k^n$
\[
  \left<\frac{x_0(a_k^n)-x_0(b_k^n)}{a_k^n-b_k^n},z_0(a_k^n)\right>=-
  \frac{1}{\|Y\|_{M_T}}\left<\frac{Y(a_k^n)-Y(b_k^n)}{a_k^n-b_k^n}
  P_{n-1}b_k^n,z_0(a_k^n)\right>.
\]
By passing to the limit as $k\to\infty$ in the previous equality
and by taking into account that $P_{n-1}b_k^n\to 0$ as
$k\to\infty$ we obtain
\[
\left<\dot x_0(\theta_0),z_0(\theta_0)\right>=0
\]
which is a contradiction, see e.g. (\cite{malb}, formula~12.9
Chap.~III). Therefore, there exist $\Delta_0>0$ and $\delta_0>0$
such that $\Gamma:B_\Delta(C_\delta)\to
\Gamma(B_\Delta(C_\delta))$ is injective for
$\Delta\in(0,\Delta_0]$ and $\delta\in(0,\delta_0].$ Let us show
that $\Delta_0>0$ and $\delta_0>0$ can be chosen also in such a
way that
\begin{equation}\label{need2}
\Gamma(B_\Delta(C_\delta)) {\rm\ \ is\ open\ in\ }
\mathbb{R}^n{\rm\  for\ any\ } \Delta\in(0,\Delta_0] {\rm\  and\
any\ }\delta\in(0,\delta_0].
\end{equation}
Observe that for any $\zeta\in\mathbb{R}^n$ satisfying
$P_{n-1}\zeta=0$ we have
\[\Gamma'(\zeta)=\frac{1}{\|Y\|_{M_T}}\Phi\left(\zeta^n+\overline
{\theta}\right)\left(\begin{array}{cc} {\rm
e}^{\Lambda(\zeta^n+\overline{\theta})} & 0 \\ 0 & 0
\end{array}\right)+
\left(0\ \ldots\ 0\ \ \dot x_0(\zeta^n+\overline{\theta})\right)
\]
and so for any $\zeta\in\mathbb{R}^n$ such that $P_{n-1}\zeta=0$
the derivative $\Gamma'(\zeta)$ is invertible. Therefore, without
loss of generality, we may consider $\Delta_0>0$ and $\delta_0>0$
sufficiently small to have that $\Gamma'(\zeta)$ is invertible for
any $\zeta\in B_\Delta(C_\delta)$ with $\Delta\in(0,\Delta_0]$ and
$\delta\in(0,\delta_0].$ By the inverse map theorem, see e.g.
(\cite{rud}, Theorem 9.17) we have that $\Gamma$ is locally
invertible in $B_\Delta(C_\delta)$ with $\Delta\in(0,\Delta_0]$
and $\delta\in(0,\delta_0],$ which implies that it maps any
sufficiently small neighborhood of $\zeta$ in $\mathbb{R}^n$ into
an open set of $\mathbb{R}^n,$ which in turn implies
(\ref{need2}). Moreover, from the inverse map theorem we have also
that $\Gamma^{-1}$ is continuously differentiable in
$\Gamma(B_\Delta(C_\delta)).$
\end{proof}
We can now prove Theorem \ref{thm1}.
\begin{proof}  First of all observe
that if $x$ is a solution of the equation $x=Q_\varepsilon x$ then
$u(t)=x^{-1}(t,x(t))$ is a solution of the equation
$u=G_\varepsilon u,$ see e.g. (\cite{non}, formulas (13)-(19)),
where $G_\varepsilon:C([0,T],\mathbb{R}^n)\to
C([0,T],\mathbb{R}^n)$ is defined as follows
\[
  (G_\varepsilon
  u)(t)=x(T,u(T))+\varepsilon\int_0^t
  \left(x'_{(2)}(\tau,u(\tau))\right)^{-1}\phi(\tau,x(\tau,u(\tau)),
  \varepsilon)d\tau.
\]
Moreover, since for any open set $V\subset\mathbb{R}^n$ the
homeomorphism $(M x)(t)=x^{-1}(t,x(t))$ maps every neighborhood of
$W_V$ onto a neighborhood of the set
\[
  \widehat{W}_V=\left\{u\in C([0,T],\mathbb{R}^n):u(t)\in V,\ {\rm for\ any\
  }t\in[0,T]\right\},
\]
then by (\cite{krazab}, Theorem 26.4) we have that
\[
  d(I-Q_\varepsilon,W_{\Gamma(C_\delta)})=d(I-G_\varepsilon,
  \widehat{W}_{\Gamma(C_\delta)})
\]
provided that $d(I-G_\varepsilon,\widehat{W}_{\Gamma(C_\delta)})$
is defined. To show that
$d(I-G_\varepsilon,\widehat{W}_{\Gamma(C_\delta)})$ is defined and
to evaluate it, we introduce the vector field $A_\varepsilon:
\Gamma(B_\Delta(C_\delta)) \to\mathbb{R}^n$ as follows
\[
A_\varepsilon(\xi)=x'_{(2)}\left(T-\varepsilon
f\left(\left[\Gamma^{-1}(\xi)\right]^n\right),
x_0\left(\left[\Gamma^{-1}
  (\xi)\right]^n+\overline{\theta}\right)
  \right)\left(\xi-x_0\left(\left[\Gamma^{-1}(\xi)\right]^n+
  \overline{\theta}\right)\right)+
\]
\[
  +\,x_0\left(\left[\Gamma^{-1}(\xi)\right]^n+\overline{\theta}-
  \varepsilon f\left(   \left[\Gamma^{-1}(\xi)\right]^n
  \right)\right),
\]
where $\Gamma, \Delta, \delta>0$ are given by Lemma  \ref{lemma2}
and $f:\mathbb{R}\to\mathbb{R}$ is defined as
\begin{equation}\label{functionf}
f(t)=\left\{   \begin{array}{cl} \ \ |t|, & {\rm \ \ if\ }
f_{x_0}(\theta_1)<0{\rm\ and\ }f_{x_0}(\theta_2)<0, \\
-|t|, & {\rm \ \ if\ } f_{x_0}(\theta_1)>0{\rm\ and\
}f_{x_0}(\theta_2)>0, \\
-\,d_{\mathbb{R}}\left(f_{x_0},\left(\theta_1,\theta_2\right)\right)
\cdot t, & {\rm \ \ otherwise.}
\end{array} \right.
\end{equation}
We now prove that there exists $\delta_0>0$ such that for any
$\delta\in(0,\delta_0]$ and any
$\varepsilon\in(0,\delta^{1+\alpha}]$ both the topological degrees
$d(I-G_\varepsilon,\widehat{W}_{\Gamma(C_{\delta})})$ and
$d_{\mathbb{R}^n}(I-A_\varepsilon,\Gamma(C_{\delta}))$ are defined
and
\begin{equation}\label{co2}
  d(I-G_\varepsilon,\widehat{W}_{\Gamma(C_{\delta})})=d_{\mathbb{R}^n}
  (I-A_\varepsilon,\Gamma(C_{\delta})).
\end{equation}
To do this we introduce an auxiliary vector field
$\widehat{A}_\varepsilon:C([0,T],\mathbb{R}^n)\to
C([0,T],\mathbb{R}^n)$ by letting $(\widehat{A}_\varepsilon
u)(t)={A}_\varepsilon (u(T))$ for any $t\in[0,T]$ and any $u\in
C([0,T],\mathbb{R}^n).$ Since
$\widehat{W}_{\Gamma(C_{\delta})}\cap\mathbb{R}^n=\Gamma(C_{\delta}),$
by the reduction theorem for the topological degree, see e.g.
(\cite{krazab}, Theorem 27.1),
$d_{\mathbb{R}^n}(I-A_\varepsilon,\Gamma(C_{\delta}))$ is defined
provided that
$d(I-\widehat{A}_\varepsilon,\widehat{W}_{\Gamma(C_{\delta})})$ is
defined, moreover
$d_{\mathbb{R}^n}(I-A_\varepsilon,\Gamma(C_{\delta}))=
d(I-\widehat{A}_\varepsilon,\widehat{W}_{\Gamma(C_{\delta})}).$
Hence, we now show that there exists $\delta_0>0$ such that for
any $\delta\in(0,\delta_0]$ and any
$\varepsilon\in(0,\delta^{1+\alpha}]$ both the Leray-Schauder
topological degrees
$d(I-G_\varepsilon,\widehat{W}_{\Gamma(C_{\delta})})$ and
$d(I-\widehat{A}_\varepsilon,\widehat{W}_{\Gamma(C_{\delta})})$
are defined and
\begin{equation}\label{co2p}
d(I-G_\varepsilon,\widehat{W}_{\Gamma(C_{\delta})})=
d(I-\widehat{A}_\varepsilon,\widehat{W}_{\Gamma(C_{\delta})}).
\end{equation}
To prove (\ref{co2p}) let $F_\varepsilon:C([0,T],\mathbb{R}^n)\to
C([0,T],\mathbb{R}^n)$ be the operator given by
\[({F_\varepsilon}u)(t)=\int_0^t
  \left(x'_{(2)}(\tau,u(\tau))\right)^{-1}\phi(\tau,x(\tau,u(\tau)),
  \varepsilon)d\tau \quad \mbox{for any} \; t\in[0,T],\]
and introduce the linear deformation
\[
  D_\varepsilon\big(\lambda,u)(t)=\lambda\Big(u(t)-x(T,u(T))-\varepsilon
  ({F_\varepsilon}u)(t)\Big)+(1-\lambda)\left(u(t)-
  \left(\widehat{A}_\varepsilon u\right)(t)\right),
\]
where $\lambda\in[0,1],\  u\in\partial
\widehat{W}_{\Gamma(C_\delta)},\
  \delta\in(0,\delta_0)$. Equivalently,
\begin{equation}
\begin{aligned}
D_\varepsilon(\lambda,u)(t)=&\lambda\Big(u(t)-x(T,u(T))\Big)+
(1-\lambda)u(t)\nonumber\\
   &-(1-\lambda)x'_{(2)}\left(T-\varepsilon f\left(\left[\Gamma^{-1}(u(T))\right]^n\right),
   \mathcal{P}_{x_0}(u(T))\right)(u(T)-\mathcal{P}_{x_0}(u(T)))\nonumber\\
   &- \lambda\varepsilon({F_\varepsilon}u)(t)-(1-\lambda)\,x_0\left(\left[\Gamma^{-1}(u(T))\right]^n+
  \overline{\theta}-\varepsilon f\left(   \left[\Gamma^{-1}(u(T))\right]^n
  \right)\right),
\end{aligned}
\end{equation}
where $\lambda\in[0,1],\ u\in\partial
\widehat{W}_{\Gamma(C_\delta)},\ \delta\in(0,\delta_0)$ and
$$\mathcal{P}_{x_0}(\xi)=x_0\left(\left[\Gamma^{-1}(\xi)\right]^n+\overline{\theta}\right).$$
We show that for all sufficiently small $\delta\in(0,\delta_0]$
and $\varepsilon\in(0,\delta^{1+\alpha}]$ we have that
$D_\varepsilon(\lambda,u)\not=0$ for any $\lambda\in[0,1]$ and any
$u\in\partial \widehat{W}_{\Gamma(C_{\delta})}.$ Assume the
contrary, thus there exist
$\{\delta_k\}_{k\in\mathbb{N}}\subset\mathbb{R}_+,$ $\delta_k\to
0$ as $k\to\infty,$ $\{\varepsilon_k\}_{k\in\mathbb{N}},$
$\varepsilon_k\in(0,\delta_k^{1+\alpha}),$
$\{u_k\}_{k\in\mathbb{N}},$ $u_k\in\partial
\widehat{W}_{\Gamma(C_{\delta_k})},$
 $\{\lambda_k\}_{k\in\mathbb{N}}\subset[0,1]$ such that
\begin{equation}
\begin{aligned}
0=\, &\lambda_k\Big(u_k(t)-x(T,u_k(T))\Big)+(1-\lambda_k)u_k(t)
\\ & -(1-\lambda_k)x'_{(2)}\left(T-\varepsilon_k f\left(\left[\Gamma^{-1}(u_k(T))\right]^n\right),
\mathcal{P}_{x_0}(u_k(T))\right)(u_k(T)-\mathcal{P}_{x_0}(u_k(T)))
\\ & -\lambda_k\varepsilon_k({F_{\varepsilon_k}}u_k)(t)-(1-\lambda_k)
x_0\left(\left[\Gamma^{-1}(u_k(T))\right]^n+\overline{\theta}-
\varepsilon_k f\left(   \left[\Gamma^{-1}(u_k(T))\right]^n
\right)\right).\label{fir}
\end{aligned}
\end{equation}
From (\ref{fir}) we have
\begin{equation}
\begin{aligned}
  u_k(t)=&\lambda_k x(T,u_k(T))\nonumber\\&+(1-\lambda_k)x'_{(2)}
  \left(T-\varepsilon_k f\left(\left[\Gamma^{-1}(u_k(T))\right]^n\right),\mathcal{P}_{x_0}(u_k(T))\right)
  (u_k(T)-\mathcal{P}_{x_0}(u_k(T)))\nonumber\\
  & +\lambda_k\varepsilon_k({F_{\varepsilon_k}}u_k)(t)+(1-\lambda_k)
  x_0\left(\left[\Gamma^{-1}(u_k(T))\right]^n+\overline{\theta}-
  \varepsilon_k f\left(   \left[\Gamma^{-1}(u_k(T))\right]^n
  \right)\right)\nonumber
\end{aligned}
\end{equation}
and therefore
\begin{equation}\label{cons}
\dot u_k(t)=\lambda_k\varepsilon_k
\left(x'_{(2)}(t,u_k(t))\right)^{-1}\phi(t,x(t,u_k(t)),\varepsilon_k).
\end{equation}
It follows from (\ref{cons}) that without loss of generality we
may assume that there exists $\xi_0\in\mathbb{R}^n$ such that
\[
  u_k(t)\to \xi_0{\ \rm as\ }k\to\infty
\]
uniformly with respect to $t\in[0,T].$ Since
$u_k(0)\in\Gamma(C_{\delta_k})\in
B_{\delta_k}(x_0([\theta_1,\theta_2]))$ then $\xi_0\in
x_0([\theta_1,\theta_2]).$ Now, to get a contradiction, take $t=T$
and rewrite (\ref{fir}) as follows
\begin{equation}\label{kmn}
\begin{aligned}
     0=\,&\lambda_k\left(u_k(T)-x(T,u_k(T))\right)+(1-\lambda_k)u_k(T)\\
    &-(1-\lambda_k)x'_{(2)}\left(T-\varepsilon_k f\left(\left[\Gamma^{-1}(u_k(T))\right]^n\right),
   \mathcal{P}_{x_0}(u_k(T))\right)(u_k(T)-\mathcal{P}_{x_0}(u_k(T)))\\
    &-\lambda_k\varepsilon_k({F_{\varepsilon_k}}u_k)(T)-(1-\lambda_k)x_0
  \left(\left[\Gamma^{-1}(u_k(T))\right]^n+\overline{\theta}-
  \varepsilon_k f\left(   \left[\Gamma^{-1}(u_k(T))\right]^n
  \right)\right)\\
   =\,&\lambda_k\left(u_k(T)-x(T,u_k(T))\right)+(1-\lambda_k)
 \left(I-x'_{(2)}\left(T-\varepsilon_k f\left(\left[\Gamma^{-1}(u_k(T))\right]^n\right),
 \mathcal{P}_{x_0}(u_k(T))\right)\right)\\
  &\cdot (u_k(T)-\mathcal{P}_{x_0}(u_k(T)))-\lambda_k\varepsilon_k({F_{\varepsilon_k}}u_k)(T)+(1-\lambda_k)
  \mathcal{P}_{x_0}(u_k(T))\\
   &-(1-\lambda_k)x_0\left(\left[\Gamma^{-1}(u_k(T))\right]^n+
  \overline{\theta}-\varepsilon_k f\left(\left[\Gamma^{-1}(u_k(T))\right]^n
  \right)\right)=\lambda_k\left(u_k(T)-x(T,u_k(T))\right) \\
   &+(1-\lambda_k)
  \left(I-x'_{(2)}\left(T-\varepsilon_k f\left(\left[\Gamma^{-1}(u_k(T))\right]^n\right),
  \mathcal{P}_{x_0}(u_k(T))\right)\right)(u_k(T)-\mathcal{P}_{x_0}(u_k(T)))\\
   &-\lambda_k\varepsilon_k ({F_{\varepsilon_k}}u_k)(T)+\varepsilon_k(1-\lambda_k)
  \dot x_0\left(\left[\Gamma^{-1}(u_k(T))\right]^n+\overline{\theta}\right)
  f\left(   \left[\Gamma^{-1}(u_k(T) )\right]^n
  \right)+o(\varepsilon_k).
\end{aligned}
\end{equation}
Now, observing that
\begin{equation}
\begin{aligned}
x(T,\xi)-\xi
&= x(T,\xi)-\mathcal{P}_{x_0}(\xi)+\mathcal{P}_{x_0}(\xi)-\xi\nonumber\\
&=x(T,\mathcal{P}_{x_0}(\xi)+(\xi-\mathcal{P}_{x_0}(\xi)))-\mathcal{P}_{x_0}(\xi)+
\mathcal{P}_{x_0}(\xi)-\xi\nonumber\\
&=x'_{(2)}(T,\mathcal{P}_{x_0}(\xi))(\xi-\mathcal{P}_{x_0}(\xi))-(\xi-\mathcal{P}_{x_0}(\xi))+
o(\xi-\mathcal{P}_{x_0}(\xi))\nonumber\\
&=\left(x'_{(2)}(T,\mathcal{P}_{x_0}(\xi))-I\right)(\xi-\mathcal{P}_{x_0}(\xi))
+o(\xi-\mathcal{P}_{x_0}(\xi)),
\end{aligned}
\end{equation}
from (\ref{kmn}) we obtain
\begin{equation}\label{cc}
\begin{aligned}
  & \lambda_k\left(I-x'_{(2)}(T,\mathcal{P}_{x_0}(u_k(T)))\right)(u_k(T)-\mathcal{P}_{x_0}(u_k(T)))
-\lambda_k \, o(u_k(T)-\mathcal{P}_{x_0}(u_k(T)))
\\
  & +(1-\lambda_k)
  \left(I-x'_{(2)}\left(T-\varepsilon_k f\left(\left[\Gamma^{-1}(u_k(T))\right]^n\right),
  \mathcal{P}_{x_0}(u_k(T))\right)\right)(u_k(T)-\mathcal{P}_{x_0}(u_k(T)))
\\
  & -\lambda_k\varepsilon_k ({F_{\varepsilon_k}}u_k)(T)+\varepsilon_k(1-\lambda_k)
  \dot x_0\left(\left[\Gamma^{-1}(u_k(T))\right]^n+\overline{\theta}\right)
  f\left(   \left[\Gamma^{-1}(u_k(T) )\right]^n
  \right)+o(\varepsilon_k)=0.
\end{aligned}
\end{equation}
We may assume that the sequences $\{\lambda_k\}_{k\in\mathbb{N}}$
and
$\left\{\dfrac{u_k(T)-\mathcal{P}_{x_0}(u_k(T))}{\|u_k(T)-\mathcal{P}_{x_0}(u_k(T))\|}
\right\}_{k\in\mathbb{N}}$ converge, let
$\lambda_0=\lim_{k\to\infty}\lambda_k$ and
$l_0=\lim_{k\to\infty}\dfrac{u_k(T)-\mathcal{P}_{x_0}(u_k(T))}{\|u_k(T)-\mathcal{P}_{x_0}(u_k(T))\|}.$
Since $u_k\in\partial\widehat{W}_{\Gamma(C_{\delta_k})}$ then
there exists $t_k\in[0,T]$ such that $u_k(t_k)\in\partial
\Gamma(C_{\delta_k}).$ Let $\zeta_k=\Gamma^{-1}(u_k(t_k)),$
without loss of generality we may assume that either
\begin{equation}\label{case2}
\zeta_k^n+\overline{\theta}\in(\theta_1,\theta_2)\quad{\rm for\
any} \; k\in\mathbb{N}
\end{equation}
or
\begin{equation}\label{case1}
  \zeta_k^n+\overline{\theta}\in\{\theta_1\}\cup\{\theta_2\}\quad{\rm for\ any}
\;k\in\mathbb{N}.
\end{equation}

\vskip0.2truecm \noindent Let us show that (\ref{case2}) cannot
occur. By Lemma \ref{lemma2}, $\Gamma$ is  a homeomorphism of
$B_\Delta(C_{\delta_k})$ onto $\Gamma(B_\Delta(C_{\delta_k}))$ for
sufficiently small $\Delta>0$ and
$u_k(t_k)\in\partial\Gamma(C_{\delta_k})$ then we have
\begin{equation}\label{tm}
\zeta_k=\Gamma^{-1}(u_k(t_k))\in\partial C_{\delta_k}.
\end{equation}
Hence (\ref{case2}) and (\ref{tm}) imply
\begin{equation}\label{imp}
  \|P_{n-1}\zeta_k\|=\delta_k\quad{\rm for\ any\
  }k\in\mathbb{N}.
\end{equation}
Since
\[
  \|P_{n-1}\zeta_k\|=\|Y^{-1}(\theta)Y(\theta)P_{n-1}\zeta_k\|\le\|Y^{-1}
  (\theta)\| \|Y(\theta)P_{n-1}\zeta_k\|
\]
then there exists $c>0$ such that
\[
  \|Y(\theta)P_{n-1}\zeta_k\|\ge c\|P_{n-1}\zeta_k\|=c\delta_k
\]
for any $\theta\in [0,T]$, and so we have
\begin{equation}\label{he1}
 \|u_k(t_k)-\mathcal{P}_{x_0}(u_k(t_k))\|=\|\Gamma(\zeta_k)-x_0(\zeta_k^n+
 \overline{\theta})\|=
 \left\|
 Y\left(\zeta_k^n+\overline{\theta}\right)P_{n-1}\zeta_k
 \right\|\ge c\delta_k
\end{equation}
for any $k\in\mathbb{N}$. On the other hand from (\ref{cons}) we
have that there exists $c_1>0$ such that
\begin{equation}\label{he2}
  \|u_k(T)-u_k(t_k)\|\le c_1\varepsilon_k\quad {\rm for\ any \
  }k\in\mathbb{N}.
\end{equation}
Finally, from Lemma \ref{lemma2} we have that
$x_0\left(\left[\Gamma^{-1}(\cdot)\right]^n+\overline{\theta}\right)$
is continuously differentiable and so by taking into account
(\ref{he2}) there exists $c_2>0$ such that
\begin{equation}
\begin{aligned}
  &\|\mathcal{P}_{x_0}(u_k(T))-\mathcal{P}_{x_0}(u_k(t_k))\|\\
  &=\left\|  x_0\left(\left[\Gamma^{-1}(u_k(T))\right]^n+\overline{\theta}\right)-x_0\left
  (\left[\Gamma^{-1}(u_k(t_k))\right]^n+\overline{\theta}\right)\right\|
\\ & \le  c_2\|u_k(T)-u_k(t_k)\|\le c_1 c_2\varepsilon_k\quad {\rm for\ any \
  }k\in\mathbb{N}.\label{he3}
\end{aligned}
\end{equation}
We are now in a position to estimate
$\|u_k(T)-\mathcal{P}_{x_0}(u_k(T))\|$ from below. We have
\begin{equation}
\begin{aligned}
& \|u_k(T)-\mathcal{P}_{x_0}(u_k(T))\|
\\&=\|u_k(t_k)-\mathcal{P}_{x_0}(u_k(t_k))+u_k(T)-
u_k(t_k)-(\mathcal{P}_{x_0}(u_k(T))-\mathcal{P}_{x_0}
(u_k(t_k)))\|\\
&\ge\left|\|u_k(t_k)-\mathcal{P}_{x_0}(u_k(t_k))\|-\|u_k(T)-u_k(t_k)-(\mathcal{P}_{x_0}(u_k(T))-
\mathcal{P}_{x_0}(u_k(t_k)))\|\right|\label{he5}.
\end{aligned}
\end{equation}
Since $\varepsilon_k\in(0,\delta_k^{1+\alpha})$  there exists
$k_0\in\mathbb{N}$ such that $c_1\varepsilon_k+c_1
c_2\varepsilon_k<c\delta_k$ for all $k\ge k_0.$ Therefore, from
(\ref{he2}) and (\ref{he3}) we have
\begin{equation}\label{he4}
\|u_k(T)-u_k(t_k)-(\mathcal{P}_{x_0}(u_k(T))-\mathcal{P}_{x_0}(u_k(t_k)))\|\le
 c_1\varepsilon_k+c_1
c_2\varepsilon_k< c\delta_k,
\end{equation}
for any $k\ge k_0$. By using (\ref{he1}) and (\ref{he4}) we may
rewrite (\ref{he5}) as follows
\begin{equation}
\begin{aligned}
&\|u_k(T)-\mathcal{P}_{x_0}(u_k(T))\|\\
&\ge \|u_k(t_k)-\mathcal{P}_{x_0}(u_k(t_k))\|
\\&-\|u_k(T)-u_k(t_k)-(\mathcal{P}_{x_0}(u_k(T))-
\mathcal{P}_{x_0}(u_k(t_k)))\|\label{he6}
\end{aligned}
\end{equation}
and so
\[
  \|u_k(T)-\mathcal{P}_{x_0}(u_k(T))\|\ge c\delta_k - c_1\varepsilon_k-c_1
c_2\varepsilon_k\quad{\rm for\ any\ }k\ge k_0.
\]
By using this inequality  we obtain for any $k\ge k_0$
\begin{equation}\label{pro}
\begin{aligned}
  & \frac{\varepsilon_k}{\|u_k(T)-\mathcal{P}_{x_0}(u_k(T))\|}\le
  \frac{\varepsilon_k}{c\delta_k - c_1\varepsilon_k-c_1
  c_2\varepsilon_k}\le
\\ & \le \frac{\delta_k^{1+\alpha}}{c\delta_k- c_1
\delta_k^{1+\alpha}-c_1 c_2
\delta_k^{1+\alpha}}=\frac{\delta_k^\alpha}{c-c_1
\delta_k^\alpha-c_1 c_2 \delta_k^\alpha}.
\end{aligned}
\end{equation}
Using (\ref{pro}) and passing to the limit as $k\to\infty$ in
(\ref{cc}) divided by $\|u_k(T)-\mathcal{P}_{x_0}(u_k(T))\|$ we
get
\begin{equation}\label{res}
  \left(I-x'_{(2)}(T,x_0(\zeta_0^n+\overline{\theta}))\right)l_0=0.
\end{equation}
In order to prove that (\ref{res}) leads to a contradiction we now
show that
\begin{equation}\label{lus}
  \left<(I-x'_{(2)}(T,\xi_0))l_0,z_0\left(\left[\Gamma^{-1}(\xi_0)\right]^n+
  \overline{\theta}\right)\right>=0.
\end{equation}
Indeed
\begin{equation}
\begin{aligned}
 & \left<\frac{u_k(T)-\mathcal{P}_{x_0}(u_k(T))}{\|u_k(T)-\mathcal{P}_{x_0}(u_k(T))\|},z_0
  \left(\left[\Gamma^{-1}(u_k(T))\right]^n+\overline{\theta}\right)\right>\nonumber\\
&= \frac{1}{\|u_k(T)-\mathcal{P}_{x_0}(u_k(T))\|}
   \left<\Gamma(\Gamma^{-1}(u_k(T)))-x_0\left(\left[\Gamma^{-1}(u_k(T))\right]^n+
   \overline{\theta}\right),z_0\left(\left[\Gamma^{-1}(u_k(T))\right]^n+
   \overline{\theta}\right)\right>\nonumber\\
 &= \frac{1}{\|u_k(T)-\mathcal{P}_{x_0}(u_k(T))\|}\left<Y\left(\left[\Gamma^{-1}
  (u_k(T))\right]^n+\overline{\theta}\right)P_{n-1}\Gamma^{-1}(u_k(T)),z_0
  \left(\left[\Gamma^{-1}(u_k(T))\right]^n+\overline{\theta}\right)\right>,
\end{aligned}
\end{equation}
and so by Lemma  \ref{lemma1} we can conclude that
\begin{equation}\label{willu}
  \left<\frac{u_k(T)-\mathcal{P}_{x_0}(u_k(T))}{\|u_k(T)-\mathcal{P}_{x_0}(u_k(T))\|},z_0
  \left(\left[\Gamma^{-1}(u_k(T))\right]^n+\overline{\theta}\right)\right>=
  0\quad{\rm for\ any\ }k\in\mathbb{N}.
\end{equation}
By the definition of the vector $l_0$ from (\ref{willu}), passing
to the limit as $k\to\infty$, we obtain
\begin{equation}\label{willu1}
  \left<l_0,z_0\left(\zeta^n_0+\overline{\theta}\right)\right>=0.
\end{equation}
Since $\|l_0\|=1$ and so $l_0\not=0$, from Lemma \ref{lemma1} we
have that there exists $l_*\not=0$ such that
\begin{equation}\label{lstar}
l_0=Y\left(\zeta^n_0+\overline{\theta}\right)P_{n-1}l_*\quad{\rm
and}\quad P_{n-1}l_*=l_*,
\end{equation}
observing that, see e.g.  (\cite{kraop}, Theorem 2.1),
\begin{equation}\label{xylink}
  x'_{(2)}(t,x_0(\tau))=Y(t+\tau)Y^{-1}(\tau),\quad {\rm for\ any\
  }t,\tau\in\mathbb{R}
\end{equation}
we have
\begin{equation}\label{abc}
\begin{aligned}
   &
  \left(I-x'_{(2)}(T,x_0(\zeta_0^n+\overline{\theta}))\right)l_0= \left(I-Y\left(T+\zeta_0^n+\overline{\theta}\right)
   Y^{-1}\left(\zeta_0^n+\overline{\theta}\right)\right)l_0
\\
    &= \left(Y\left(\zeta_0^n+\overline{\theta}\right)-Y\left(T+\zeta_0^n+\overline{\theta}\right)
   \right)P_{n-1}l_*
\\
   &= \Phi\left(\zeta_0^n+\overline{\theta}\right)\left(\left(
    \begin{array}{ll} {\rm
    e}^{\Lambda(\zeta_0^n+\overline{\theta})}& 0 \\
    0 & 1\end{array}\right)-\left(\begin{array}{ll} {\rm
    e}^{\Lambda(T+\zeta_0^n+\overline{\theta})}& 0 \\
    0 & 1\end{array}\right)\right)P_{n-1}l_*\\
    &= \Phi\left(\zeta_0^n+\overline{\theta}\right)
    \left(
    \begin{array}{ll} {\rm
    e}^{\Lambda(\zeta_0^n+\overline{\theta})} (I-{\rm e}^{\Lambda T})& 0 \\
    0 & 0\end{array}\right)P_{n-1}l_*
\end{aligned}
\end{equation}
contradicting (\ref{res}).

\vskip0.2truecm \noindent Let us now show that (\ref{case1}) also
cannot occur. Firstly observe that if, passing to a subsequence if
necessary, we have that
$\dfrac{\varepsilon_k}{\|u_k(T)-\mathcal{P}_{x_0}(u_k(T))\|} \to
0$ then we can proceed as before to obtain again (\ref{res}) and
so a contradiction. Therefore, consider the case when
$\dfrac{\varepsilon_k}{\|u_k(T)-\mathcal{P}_{x_0}(u_k(T))\|} \to
l$, with $l>0$ or $l=+\infty.$ From (\ref{cc}) we have that
\begin{equation}\label{caca}
\begin{aligned}
& \frac{\varepsilon_k}{\|u_k(T)-\mathcal{P}_{x_0}(u_k(T))\|
}\left<\Xi_k(x_0)(T),z_0\left(\left[\Gamma^{-1}(u_k(T))\right]^n+\overline{\theta}\right)\right>
\\
&=\left<\Upsilon_k(x_0)(T),z_0\left(\left[\Gamma^{-1}(u_k(T))\right]^n+
\overline{\theta}\right)\right>,
\end{aligned}
\end{equation}
where
\[
 \Xi_k(x_0)(T):=\lambda_k
({F_{\varepsilon_k}}u_k)(T)-(1-\lambda_k)\dot x_0
  \left(\left[\Gamma^{-1}(u_k(T))\right]^n+\overline{\theta}\right)
  f\left(   \left[\Gamma^{-1}(u_k(T) )\right]^n
  \right)+\frac{o(\varepsilon_k)}{\varepsilon_k},
\]
\begin{equation}
\begin{aligned}
\Upsilon_k(x_0)(T):=&
\lambda_k\left(I-x'_{(2)}(T,\mathcal{P}_{x_0}(u_k(T)))\right)\frac{u_k(T)-\mathcal{P}_{x_0}(u_k(T)))}
{\|u_k(T)-\mathcal{P}_{x_0}(u_k(T))\|} -\nonumber\\
 &-\lambda_k\frac{o(u_k(T)-\mathcal{P}_{x_0}(u_k(T)))}{\|u_k(T)-\mathcal{P}_{x_0}(u_k(T))\|}+
(1-\lambda_k)\frac{u_k(T)-\mathcal{P}_{x_0}(u_k(T))}
{\|u_k(T)-\mathcal{P}_{x_0}(u_k(T))\|}\cdot\nonumber\\
&\cdot \left(I-x'_{(2)}\left(T-\varepsilon_k
f\left(\left[\Gamma^{-1}(u_k(T))\right]^n\right),\mathcal{P}_{x_0}(u_k(T))\right)\right).
\end{aligned}
\end{equation}
By using (\ref{lstar}), (\ref{abc}) and Lemma \ref{lemma1} we
obtain
\[
\left<\left(I-x'_{(2)}\left(T,x_0\left(\zeta_0^n+\overline{\theta}\right)\right)\right)l_0,
    z_0\left(\zeta_0^n+\overline{\theta}\right)\right>=
    \left<Y\left(\zeta_0^n+\overline{\theta}\right)\left(I-{\rm e}^{\Lambda T}\right)
    P_{n-1}l_*,
    z_0\left(\zeta_0^n+\overline{\theta}\right)\right>
\]
\[
    =\left<Y\left(\zeta_0^n+\overline{\theta}\right)
    P_{n-1}\left(I-{\rm e}^{\Lambda T}\right)l_*,
    z_0\left(\zeta_0^n+\overline{\theta}\right)\right>=0.
\]
Therefore
$$
\left<\Upsilon_k(x_0)(T),z_0\left
  (\left[\Gamma^{-1}(u_k(T))\right]^n+\overline{\theta}\right)\right>\
  \to 0\quad{\rm as\ }k\to\infty
$$
and from (\ref{caca}) we conclude that
\[
\left<\Xi_k(x_0)(T),z_0\left(\left[\Gamma^{-1}(u_k(T))\right]^n+
\overline{\theta}\right)\right> \to 0\quad{\rm as\ }k\to\infty
\]
which imply
\begin{equation}\label{fa}
  \left<
\lambda_0\widehat{
F}\left(x_0\left(\zeta^n_0+\overline{\theta}\right)\right)-(1-\lambda_0)\dot
x_0\left(\zeta^n_0+\overline{\theta}\right)
  f\left(   \zeta^n_0
  \right),z_0\left(\zeta^n_0+\overline{\theta}\right)\right>=0,
\end{equation}
where
\[
\widehat {F}(\xi)=\int\limits_0^T
  \left(x'_{(2)}(\tau,\xi)\right)^{-1}\phi(\tau,x(\tau,\xi),0)d\tau.
\]
By Perron's lemma we have
\[
\left< \dot x_0\left(\zeta^n_0+\overline{\theta}\right)
  f\left(   \zeta^n_0
  \right),z_0\left(\zeta^n_0+\overline{\theta}\right)\right>=
\left<\dot x_0(0),z_0(0)\right>f\left(   \zeta^n_0
  \right)
\]
and so (\ref{fa}) can be rewritten as
\begin{equation}\label{cbr}
  \lambda_0\,{\rm sign}\left<\dot x_0(0),z_0(0)\right>\left<
\widehat {F}\left(x_0(\zeta^n_0+\overline{\theta})\right),
z_0(\zeta^n_0+\overline{\theta})\right>
-(1-\lambda_0)\left|\left<\dot x_0(0),z_0(0)\right>\right|f\left(
\zeta^n_0
  \right) =0,
\end{equation}
let us show that
\begin{equation}\label{gf}
{\rm sign}\left<\dot x_0(0),z_0(0)\right>\left< \widehat{
F}\left(x_0\left(\theta\right)\right),
z_0\left(\theta\right)\right>=f_{x_0}(\theta)\quad{\rm for\ any\
}\theta\in[0,T].
\end{equation}
Denote by $Z(t)$ and $Z_0(t)$ the fundamental matrixes of the
adjoint system (\ref{as}) such that $Z(0)=I$ and
$Z_0(t)=(Z_{n-1}(t)\ z_0(t)),$ where $Z_{n-1}(t)$ is a $n\times
n-1$ matrix whose columns are (not $T$-periodic) linearly
independent solutions of (\ref{as}). Since
\[
  \left(x'_{(2)}(\tau,x_0(\theta))\right)^{-1}=Y(\theta)Y^{-1}(\tau+\theta)=
  \left(Z^{-1}(\theta)\right)^* Z^*(\tau+\theta)=\left(Z^{-1}_0(\theta)
  \right)^* Z^*_0(\tau+\theta),
\]
see e.g. (\cite{dem}, Chap.~III \S 12), and
$z_0(\theta)=\left(Z_{n-1}(\theta)\ \ z_0(\theta)\right)\left(\begin{array}{c} 0\\ \vdots\\ 0 \\
1\end{array}\right)$ then we have
\begin{equation}
\begin{aligned}
&\left< \widehat {F}\left(x_0\left(\theta\right)\right),
z_0\left(\theta\right)\right>=\left<\left(Z^{-1}_0(\theta)\right)
^*\int\limits_0^T
Z^*_0(\tau+\theta)\,\phi(\tau,x_0(\tau+\theta),0)d\tau,
z_0\left(\theta\right)\right>\\ \nonumber
&=\left<\int\limits_\theta^{T+\theta} \left(\begin{array}{c}
Z_{n-1}^*(\tau)\\ z_0(\tau)\end{array}\right)
\phi(\tau-\theta,x_0(\tau),0)d\tau,
\left(\begin{array}{c} 0\\ \vdots\\ 0 \\
1\end{array}\right)\right>\\ \nonumber
&=\int\limits_\theta^{T+\theta}
 \left<z_0(\tau),\,
\phi(\tau-\theta,x_0(\tau),0)\right>d\tau=f_{x_0}(\theta)
\end{aligned}
\end{equation}
and so (\ref{gf}) holds. By taking into account (\ref{gf}) we can
finally rewrite  (\ref{cbr}) as follows
\[
  \lambda_0 f_{x_0}\left(\zeta^n_0+\overline{\theta}\right)
-(1-\lambda_0)\left|\left<\dot x_0(0),z_0(0)\right>\right|f\left(
\zeta^n_0
  \right) =0,
\]
where either $\zeta^n_0+\overline{\theta}=\theta_1$ or
$\zeta^n_0+\overline{\theta}=\theta_2.$ This can be rewritten as
\begin{equation}\label{rev}
  \lambda_0 f_{x_0}\left(\theta_i\right)
-(1-\lambda_0)\left|\left<\dot x_0(0),z_0(0)\right>\right|f\left(
(-1)^i|\zeta^n_0|
  \right) =0,
\end{equation}
where either $i=1$ or $i=2$. If
$d_{\mathbb{R}}(f_{x_0},(\theta_1,\theta_2))=0,$ then, see
(\cite{krazab}, \S 3.2) for the definition of Brouwer degree in
$\mathbb{R},$ for any $i=1,2$ and any $a\ge 0$ we have
$$
f\left((-1)^i a\right)=-a\,{\rm sign}(f_{x_0}(\theta_1))=-a\,{\rm
sign}(f_{x_0}(\theta_2))
$$ and so if
$d_{\mathbb{R}}(f_{x_0},(\theta_1,\theta_2))=0$ then (\ref{rev})
can be rewritten as
\begin{equation}\label{rew1}
  \lambda_0 f_{x_0}\left(\theta_i\right)
+(1-\lambda_0)\left|\left<\dot
x_0(0),z_0(0)\right>\right|\,|\zeta_0^n|\,{\rm sign}
\left(f_{x_0}(\theta_i)\right) =0,
\end{equation}
where either $i=1$ or $i=2$. If
$d_{\mathbb{R}}(f_{x_0},(\theta_1,\theta_2))\not=0,$ then for
$i=1,2$ and any $a\ge 0$ we have
$$ f\left((-1)^i
a\right)=(-1)^{i+1} a\,
d_{\mathbb{R}}(f_{x_0},(\theta_1,\theta_2))=(-1)^{i+1} a \,(-1)^i
{\rm sign}(f_{x_0}(\theta_i))=-a\,{\rm sign}(f_{x_0}(\theta_i))$$
and so (\ref{rev}) can be rewritten again as (\ref{rew1}). But
(\ref{rew1}) contradicts either the assumption that
$f_{x_0}(\theta_1)\not= 0$ (in the case when $i=1$) or the
assumption that $f_{x_0}(\theta_2)\not= 0$ (in the case when
$i=2$).

\vskip0.2truecm \noindent Therefore, neither (\ref{case1}) nor
(\ref{case2}) can occur and so there exists $\delta_0>0$ such that
for any $\delta\in(0,\delta_0]$ and any
$\varepsilon\in(0,\delta^{1+\alpha}]$ we have that
$D_\varepsilon(\lambda,u)\not=0$ for any $\lambda\in[0,1]$ and any
$u\in\partial\widehat{W}_{\Gamma(C_{\delta})}.$ Thus for any
$\delta\in(0,\delta_0]$ and $\varepsilon\in(0,\delta^{1+\alpha}]$
both the Leray-Schauder degrees
$d(I-G_\varepsilon,\widehat{W}_{\Gamma(C_{\delta})})$ and
$d(I-\widehat{A}_\varepsilon,\widehat{W}_{\Gamma(C_{\delta})})$
are defined and (\ref{co2p}) holds. As already noticed
(\ref{co2p}) implies (\ref{co2}), hence to finish the proof it
remains only to show that
$d(I-A_\varepsilon,\Gamma(C_\delta))=(-1)^{\beta(x_0)}
d_{\mathbb{R}}\left(f_{x_0},\left(\theta_1,\theta_2\right)\right)
$ for any $\delta\in(0,\delta_0]$ and
$\varepsilon\in(0,\delta^{1+\alpha}].$ Let $\delta\in(0,\delta_0]$
and $\varepsilon\in(0,\delta^{1+\alpha}],$ since $\Gamma$ is a
homeomorphism of $B_\Delta(C_{\delta})$ onto
$\Gamma(B_\Delta(C_{\delta}))$ by (\cite{krazab}, Theorem 26.4) we
obtain
\[
  d_{\mathbb{R}^n}(I-A_\varepsilon,\Gamma(C_\delta))=d_{\mathbb{R}^n}
  (I-\Gamma^{-1}A_\varepsilon\Gamma,C_\delta).
\]
Let $\zeta\in C_\delta.$ Taking into account (\ref{xylink}) and
(\ref{abc}) we have
\[
  \zeta-(\Gamma^{-1}A_\varepsilon\Gamma)(\zeta)=\zeta-
  (\Gamma^{-1}A_\varepsilon)\left(\frac{Y(\zeta^n+\overline{\theta})}
  {\|Y\|_{M_T}}
  P_{n-1}\zeta+x_0(\zeta^n+\overline{\theta})\right)
\]
\[
  =\zeta-\Gamma^{-1}\left(x'_{(2)}\left(
T-\varepsilon f(\zeta^n),x_0\left(\zeta^n+
  \overline{\theta}\right)\right)\frac{Y(\zeta^n+\overline{\theta})}
  {\|Y\|_{M_T}}P_{n-1}\zeta+
  x_0(\zeta^n+\overline{\theta}-\varepsilon
    f(\zeta^n))\right)
\]
\[=\zeta-\Gamma^{-1}\Bigg(\frac{Y\left(\zeta^n+\overline{\theta}-\varepsilon
f(\zeta^n)\right)}
  {\|Y\|_{M_T}}P_{n-1}\left(\begin{array}{cc}
  {\rm e}^{\Lambda T} & 0 \\
  0 &   0\end{array}\right)\zeta +x_0\left(\zeta^n+\overline{\theta}-
  \varepsilon f(\zeta^n)\right)\Bigg)
\]
\[
  =\zeta-\left(\begin{array}{c}
  {\rm e}^{\Lambda T}\left(
  {\zeta|}_{\mathbb{R}^{n-1}}\right) \\
    \zeta^n-\varepsilon f(\zeta^n)\end{array}\right)
\] and so
\[
  d_{\mathbb{R}^n}(I-\Gamma^{-1}A_\varepsilon\Gamma,C_\delta)=
  d_{\mathbb{R}^n}\left(\left(I-{\rm e}^{\Lambda T}\right)\times
  \varepsilon f,C_\delta\right),
\]
where $\left(I-{\rm e}^{\Lambda T}\right)\times
  \varepsilon f=\left(I-{\rm e}^{\Lambda T},
  \varepsilon f\right).$ By the property of the Brouwer topological degree for the
product of vector fields, see e.g. (\cite{krazab}, Theorem 7.4) we
have
\[
  d_{\mathbb{R}^n}\left(\left(I-{\rm e}^{\Lambda T}\right)\times
   \varepsilon f,C_\delta\right)=
  d_{\mathbb{R}^n}\left(I-{\rm e}^{\Lambda
  T},B_{\delta}(0)\right)\cdot
  d_{\mathbb{R}}\left(\varepsilon  f,\left(-\frac{\theta_2-\theta_1}{2},
  \frac{\theta_2-\theta_1}{2}\right)\right),
\]
where $d_{\mathbb{R}^n}\left(I-{\rm e}^{\Lambda
T},B_{\delta}(0)\right)=(-1)^{\beta(x_0)}$ by (\cite{krazab},
Theorem 6.1), and by a direct computation we have that
\[ d_{\mathbb{R}}\left(\varepsilon
  f,\left(-\frac{\theta_2-\theta_1}{2},\frac{\theta_2-
  \theta_1}{2}\right)\right)=
\,-\,d_{\mathbb{R}}\left(f_{x_0},\left(\theta_1,\theta_2\right)\right).
\]
Thus, we finally have that
\[
  d_{\mathbb{R}^n}(I-\Gamma^{-1}A_\varepsilon\Gamma,C_\delta)=\,-\,
  (-1)^{\beta(x_0)}d_{\mathbb{R}}\left(f_{x_0},\left(\theta_1,
  \theta_2\right)\right).
\]
In conclusion, we have proved that there exists $\delta_0>0$ such
that for any $\delta\in(0,\delta_0]$ and any
$\varepsilon\in(0,\delta^{1+\alpha}]$ the Leray-Schauder
topological degree $d(I-Q_\varepsilon,W_{\Gamma(C_\delta)})$ is
defined and it can be calculated by the formula
\[
d(I-Q_\varepsilon,W_{\Gamma(C_\delta)})=\,-\,(-1)^{\beta(x_0)}
d_{\mathbb{R}}\left(f_{x_0},\left(\theta_1,\theta_2\right)\right).
\]
To conclude the proof we have only to show that
$V_\delta:=\Gamma(C_\delta)$ satisfies properties 1) and 2). To
this end,  let $\xi\in\Gamma(C_\delta),$ thus
\[
  \xi=\frac{Y(\zeta^n+\overline{\theta})}{\|Y\|_{M_T}}
  P_{n-1}\zeta+x_0(\zeta^n+\overline{\theta}).
\]
for some $\zeta\in\mathbb{R}^n$ satisfying
$\|P_{n-1}\zeta\|\le\delta$ and
$\left[\Gamma^{-1}(\xi)\right]^n+\overline{\theta}\in[\theta_1,\theta_2].$
Therefore
\[
\left\|\xi-x_0\left(\left[\Gamma^{-1}(\xi)\right]^n+\overline{\theta}
\right)\right\|=
\left\|\frac{Y(\zeta^n+\overline{\theta})}{\|Y\|_{M_T}}P_{n-1}\zeta\right\|
\le\|P_{n-1}\zeta\|\le\delta
\]
and so property 1) holds. By the definition of the set
$C_\delta$ we have that for any $\delta\in(0,\delta_0)$ both the
points $\left(0,...,0,-\dfrac{\theta_2-\theta_1}{2}\right)$ and
$\left(0,...,0,\dfrac{\theta_2-\theta_1}{2}\right)$ belong to the
boundary of $C_\delta.$ Therefore, both the points $x_0(\theta_1)$
and $x_0(\theta_2)$ belong to the boundary of $\Gamma(C_\delta).$
On the other hand if $\xi=x_0(\theta),$ where
$\theta\in(\theta_1,\theta_2),$ then
\begin{equation}\label{starr}
\Gamma^{-1}(\xi)=\left(0,...,0,\theta-\overline{\theta}\right)\subset
C_\delta.
\end{equation}
Thus $\xi\in\Gamma(C_\delta)$ and property 2) is also satisfied.
The proof of Theorem \ref{thm1} is now complete.
\end{proof}

Recall that
\[
\Theta_W(x)=\left\{\theta_0\in(0,T):S_{\theta_0}\,x\in\partial W,\
S_\theta\,
  x\in W {\rm\ for\ any\ }\theta\in(0,\theta_0)\right\},
  \quad{\rm where\ }x\in\mathfrak{S}_W,
\]
\[
(S_\theta\,x)(t) = x(t+\theta) \quad{\rm and}
\]
\[
\beta(x_0)\; {\rm is\ the\ sum\ of\ the\ multiplicities\ of\ the\
characteristic\ multipliers\ greater \ than\ } 1 {\rm \ of\
}(\ref{ls}).
\]

We can prove the following result.
\begin{theorem}\label{thm2} Assume that $\mathfrak{S}_W$ is finite and
it contains only nondegenerate $T$-periodic cycles of (\ref{np}).
Assume that $f_x(0)\not=0$ for any $x\in\mathfrak{S}_W.$ Then for
every $\varepsilon>0$ sufficiently small the topological degree
$d(I-Q_\varepsilon,W)$ is defined and the following formula holds
\begin{equation}\label{kmnform}
d(I-Q_\varepsilon,W)=(-1)^n
d_{\mathbb{R}^n}(\psi,W\cap\mathbb{R}^n)-\sum_{x\in\mathfrak{S}_W:~
\Theta_W(x)\not= \emptyset}(-1)^{\beta(x)}
d_{\mathbb{R}}\left(f_x,\left(0,\min\{\Theta_W(x)\}\right)\right),
\end{equation}
\end{theorem}

\begin{proof} For any $x\in\mathfrak{S}_W$ satisfying $\Theta_W(x)\not=\emptyset$ let
$\delta_0(x)$ and $\{V_\delta(x)\}_{\delta\in(0,\delta_0(x))}$ as
given by Theorem \ref{thm1}, where $x_0:=x,$ $\theta_1:=0$ and
$\theta_2:=\min\{\Theta_W(x)\}.$ Let
$\delta_1=\min_{x\in\mathfrak{S}_W:\Theta_W(x)\not=\emptyset}\delta_0(x)>0.$
Since $f_x(0)\not=0$ for any $x\in\mathfrak{S}_W$ then by Malkin's
theorem, see \cite{mal} or (\cite{malb}, Theorem p.~387), there
exists $\delta_*\in(0,\delta_1)$ and $\varepsilon_*>0$ such that
\begin{equation}\label{mark0}
   Q_\varepsilon \widetilde {x}\not=\widetilde {x}\quad{\rm for\ any\ }
\widetilde{x}\in \overline{B_{\delta_*}
  \left(x\right)}{\rm \ whenever \ }
x\in\mathfrak{S}_W{\ \rm and\ }
  \varepsilon\in(0,\varepsilon_*).
\end{equation}
By the definition of $\mathfrak{S}_W$ from (\ref{mark0}) we have
that
\begin{equation}\label{mark}
   Q_\varepsilon \widetilde{x}\not=\widetilde{x}\quad{\rm for\ any\ }
\widetilde{x}\in \overline{B_{\delta_*}
  \left(x\right)}\cup\overline{B_{\delta_*}
  \left(S_{\min\{\Theta_W(x)\}}x\right)} {\rm \ whenever\ }
x\in\mathfrak{S}_W{\ \rm and\ }
  \varepsilon\in(0,\varepsilon_*)
\end{equation}
Let $\delta_{**}\in(0,\delta_*)$  be sufficiently small in such a
way that
\[
\left(B_{\delta_*} (x)\cup B_{\delta_*}(S_{\min\{\Theta_W(x)\}}
x)\cup
  W_{V_{\delta_{**}}(x)}\right)\backslash \overline{W}\subset B_{\delta_*}
  (x)\cup B_{\delta_*}(S_{\min\{\Theta_W(x)\}}
  x)
\]
for any $x\in \mathfrak{S}_W$, therefore by taking into account
(\ref{mark}) we have
\[
  Q_\varepsilon \widetilde{x}\not=\widetilde{x}\quad{\rm for\ any\ }
\widetilde{x}\in \left(B_{\delta_*} (x)\cup
B_{\delta_*}(S_{\min\{\Theta_W(x)\}} x)\cup
  W_{V_{\delta_{**}}(x)}\right)\backslash \overline{W},
\]
whenever $x\in \mathfrak{S}_W$ and
$\varepsilon\in(0,\varepsilon_*)$. Therefore by applying the
coincidence degree formula given by Theorem \ref{thm1} for any
$x\in\mathfrak{S}_W$ such that $\Theta_W(x)\not=\emptyset$ and any
$\varepsilon\in\left(0,\min\{\delta^{1+\alpha},\varepsilon_*\}\right)$
we have
\begin{equation}
\begin{aligned}
&d\left(I-Q_\varepsilon,\left(B_{\delta_*} (x)\cup
B_{\delta_*}(S_{\min\{\Theta_W(x)\}} x)\cup
  W_{V_{\delta_{**}}(x)}\right)\cap W\right)\\
&= d\left(I-Q_{\varepsilon},B_{\delta_*}
  (x)\cup B_{\delta_*}(S_{\min\{\Theta_W(x)\}} x)
  \cup W_{V_{\delta_{**}}(x)}\right)\\
&= d(I-Q_\varepsilon,W_{V_{\delta_{**}}(x)})=\,-\,(-1)^{\beta(x)}
d_{\mathbb{R}}\left(f_x,\left(0,\min\{\Theta_W(x)\}\right)\right).\label{fo1}
\end{aligned}
\end{equation}
Let
\[
  \mathfrak{S}_W^0=\left\{x\in\mathfrak{S}_W:{\rm there\ exists\
  }\delta_0>0{\rm\ such\ that\ }S_\delta(x)\not\in\partial W{\rm\
  for\ any\ }\delta\in(-\delta_0,0)\cup (0,\delta_0)\right\}.
\]
From (\ref{mark0}) we have that
\begin{equation}\label{fo1rev}
d\left(I-Q_\varepsilon,B_{\delta_*}(x) \cap W\right)=0\quad {\rm
for \ any\ } x\in \mathfrak{S}_W^0{\ \rm and\ any \ }
  \varepsilon\in(0,\varepsilon_*).
\end{equation}
Since any point $x\in\mathfrak{S}_W$ is a limit cycle of
(\ref{np}) and, by assumption, they are in a finite number we may
assume without loss of generality that $\delta_*>0$ is
sufficiently small to have that
\begin{equation}\label{bege}
Q_0(\widehat{ x})\not=\widehat{ x} \quad {\rm for\ any\ }\widehat{
x}\in C([0,T],\mathbb{R}^n){\rm\ such\ that\ }\widehat {x}(0)\in
B_{\delta_*}(x([0,T]))\backslash x([0,T]).
\end{equation}
Therefore we have that the boundary of the set $W\backslash
E_{\delta_*}$ where
\[
E_{\delta_*}:=\left(\bigcup_{x\in\mathfrak{S}_W:\Theta_W(x)\not=\emptyset}
\left(B_{\delta_*}
  (x)\cup B_{\delta_*}(S_{\Theta_W(x)} x)\cup
  W_{V_{\delta_{**}}(x)}\right)\cap W\right)\bigcup\left(\bigcup_{x\in\mathfrak{S}_W^0}
  B_{\delta_*}(x)\cap W\right)
\]
does not contain $T$-periodic solutions of (\ref{np}). This fact
allows us to apply Corollary~1 of \cite{maw} to obtain
\begin{equation}\label{ma}
d(I-Q_0,W\backslash E_{\delta_*})=(-1)^n d_{\mathbb{R}^n}(\psi,
E_{\delta_*}\cap\mathbb{R}^n).
\end{equation}
But from (\ref{bege}) the function  $\psi$ is nondegenerate on the
set $E_{\delta_*}\cap\mathbb{R}^n$ and from (\ref{ma}) we have
that
\begin{equation}\label{fo2}
  d(I-Q_0,W\backslash E_{\delta_*})=
  (-1)^n d_{\mathbb{R}^n}(\psi,W\cap\mathbb{R}^n).
\end{equation}
From (\ref{fo1}), (\ref{fo1rev}) and (\ref{fo2}) the conclusion of
the theorem easily follows.
\end{proof}

\begin{remark}\label{remarknew}
From \ref{kmnform}) it follows that the points of $\mathfrak{S}_W$
such that $S_\theta x\not\in\mathfrak{S}_W$ for all
$\theta\in(0,T)$ do not affect the value of $d(I-Q_\varepsilon,W)$
with $\varepsilon>0$ sufficiently small.
\end{remark}

\vskip.3truecm Let $X=\{x\in C([0,T],\mathbb{R}^n):x(0)=x(T)\}$
and let $L:{\rm dom}L\subset X\to L^1([0,T],\mathbb{R}^n)$ be the
linear operator defined by $(Lx)(\cdot)=\dot x(\cdot)$ with ${\rm
dom} L=\{x\in X: x(\cdot){\rm\ is\ absolutely\ continuous}\}.$ It
is immediate to see that $L$ is a Fredholm operator of index zero.
Let $N_\varepsilon:X\to L^1([0,T],\mathbb{R}^n)$ be the Nemitcky
operator given by $(N_\varepsilon
x)(\cdot)=\psi(x(\cdot))+\varepsilon\phi(\cdot,x(\cdot),\varepsilon).$
Thus the existence of $T$-periodic solutions for system (\ref{ps})
is equivalent to the solvability of the equation
\begin{equation}\label{coequ}
Lx=N_\varepsilon x,\quad x\in{\rm dom}L.
\end{equation}

We now provide for the coincidence degree
$D_L(L-N_\varepsilon,W\cap X)$ of $L$ and $N_\varepsilon,$ see
(\cite{mawbvp}, p.~19), a formula similar to that established in
Theorem \ref{thm1}.

\vskip0.2truecm
\begin{corollary}\label{cor1}
Assume all the conditions of
Theorem \ref{thm2}. Then for $\varepsilon>0$ sufficiently small
the coincidence degree $D_L(L-N_\varepsilon,W\cap X)$ is defined
and the following formula holds
\begin{equation}
\begin{aligned}
 D_L(L-N_\varepsilon,W\cap X)=\; &(-1)^n
d_{\mathbb{R}^n}(\psi,W\cap\mathbb{R}^n)\;\\
& -\sum_{x\in\mathfrak{S}_W: ~\Theta_W(x) \not=
\emptyset}(-1)^{\beta(x)}
d_{\mathbb{R}}\left(f_x,\left(0,\min\{\Theta_W(x)\}\right)\right).\label{kmnformD}
\end{aligned}
\end{equation}
\end{corollary}

\begin{proof} Since $d(I-Q_\varepsilon,W)$ is
defined for $\varepsilon>0$ sufficiently small then
$D_L(L-N_\varepsilon,W\cap X)$ is also defined for $\varepsilon>0$
sufficiently small, see (\cite{mawbvp}, Chap. 2 \S 2). To prove
(\ref{kmnformD}) we apply the duality principles developed in
(\cite{mawbvp}, Chap.~3). First, observe that the zeros of the
operator $R_\varepsilon:C([0,T],\mathbb{R}^n)\to
C([0,T],\mathbb{R}^n)$ defined by
\begin{equation}
\begin{aligned}
(R_\varepsilon x)(t)=&\; x(t)-x(0)-\int\limits_0^T
  \left(\psi(x(\tau))+\varepsilon\phi(\tau,x(\tau),\varepsilon)\right)d\tau-\nonumber\\
&-\int\limits_0^t\left(\psi(x(\tau))+\varepsilon\phi(\tau,x(\tau),\varepsilon)
  \right)d\tau
+t\int\limits_0^T
  \left(\psi(x(\tau))+\varepsilon\phi(\tau,x(\tau),\varepsilon)\right)d\tau
\end{aligned}
\end{equation}
coincide with the fixed points of the operator $Q_\varepsilon,$
hence $d(R_\varepsilon,W)$ is also defined for $\varepsilon>0$
sufficiently small. Therefore by (\cite{mawbvp}, Theorem III.1
with $a=1$ and $b=0$) and (\cite{mawbvp}, Theorem III.4) we have
that
\[
  d(R_\varepsilon,W)=d(I-Q_\varepsilon,W).
\]
Furthermore, by using the methods employed in (\cite{mawbvp},
Chap.~III, \S 4) for defining $D_L(L-N_\varepsilon,W\cap X)$ and
by (\cite{mawbvp}, Theorem III.7) we obtain that
\[
D_L(L-N_\varepsilon,W\cap X)=d(R_\varepsilon,W),
\]
which concludes the proof.
\end{proof}

\begin{remark}\label{remark1} If $W=W_U$ for a
suitable open set $U\subset\mathbb{R}^n$ then it is possible to
rewrite (\ref{kmnform}) and (\ref{kmnformD}) in a different way by
representing the sets $\mathfrak{S}_W$ as follows
\[
  \mathfrak{S}_W=\bigcup_{\xi\in\partial
  U:x(0,\xi)=x(T,\xi)}x(\cdot,\xi)
\]
and
\[
  \Theta_W(x)=\left\{\theta_0\in(0,T):x(\theta_0)\in\partial U,\
  x(\theta)\in U{\ \rm for\ any\ }\theta\in(0,\theta_0)\right\}.
\]
Moreover, if
\begin{equation}\label{state}
{\rm any\ Cauchy\ problem\ associated\ to\ } (\ref{ps}) {\rm\ has\
an\ unique\ solution\ defined\  in\ }[0,T],
\end{equation}
then we can introduce the Poincar\'e-Andronov operator
$\Omega_\varepsilon: \mathbb{R}^n\to \mathbb{R}^n$ in the
following way
\[
\Omega_\varepsilon(\xi)=x_\varepsilon(T,\xi),
\]
where $x_\varepsilon(\cdot,\xi)$ is the solution of (\ref{ps})
satisfying $x_\varepsilon(0,\xi)=\xi.$ In this case we can provide
an analogous result to (\ref{kmnform}) for the Brouwer topological
degree of $I-\Omega_\varepsilon$ on $U.$
\end{remark}

Indeed, we can prove the following result.

\begin{corollary}\label{cor2} Assume that condition (\ref{state}) is
satisfied. Let
\[
  \mathfrak{S}^U=\bigcup_{\xi\in\partial
  U:x(0,\xi)=x(T,\xi)}x(\cdot,\xi).
\]
Assume that $\mathfrak{S}^U$ is finite and any $T$-periodic
solution $x_0\in\mathfrak{S}^U$ is a nondegenerate limit cycle of
(\ref{np}). If
\[
f_x(0)\not=0\quad{\rm for\ any\ }x\in \mathfrak{S}^U
\]
then for all $\varepsilon>0$ sufficiently small the topological
degree $d_{\mathbb{R}^n}(I-\Omega_\varepsilon,U)$ is defined and
it can be evaluated by the formula
\begin{equation}\label{kmnformU}
d_{\mathbb{R}^n}(I-\Omega_\varepsilon,U)=(-1)^n
d_{\mathbb{R}^n}(\psi,U)-\sum_{x\in\mathfrak{S}^U:~\Theta^U(x)\not=
\emptyset}(-1)^{\beta(x)}
d_{\mathbb{R}}\left(f_x,\left(0,\min\{\Theta^U(x)\}\right)\right),
\end{equation}
where, for any  $x\in \mathfrak{S}^U $,
$\Theta^U(x)=\left\{\theta_0\in(0,T):x(\theta_0)\in\partial U,\
x(\theta)\in U{\ \it for\ any\ }\theta\in(0,\theta_0)\right\}$ and
$\beta(x)$ is the sum of the multiplicities of the characteristic
multipliers greater than $1$  of (\ref{ls}) with $x_0:=x$.
\end{corollary}

\begin{proof} From Theorem  \ref{thm2}, taking into account
Remark \ref{remark1} we have that there exists $\varepsilon_0>0$
such that for every $\varepsilon\in(0,\varepsilon_0]$ the degree
$d(I-Q_\varepsilon,W_U)$ is defined and
\begin{equation}\label{kmnformU1}
\begin{aligned}
d(I-Q_\varepsilon,W_U)=&(-1)^n
d_{\mathbb{R}^n}(\psi,W_U\cap\mathbb{R}^n)\\
&-\sum_{x\in\mathfrak{S}^U: ~\Theta^U(x)\not=
\emptyset}(-1)^{\beta(x)}
d_{\mathbb{R}}\left(f_x,\left(0,\min\{\Theta^U(x)\}\right)\right).
\end{aligned}
\end{equation}
Therefore, to prove the corollary we show that
\begin{equation}\label{show1}
d(I-Q_\varepsilon,W_U)=d_{\mathbb{R}^n}(I-\Omega_\varepsilon,U)\quad{\rm
for\ any\ }\varepsilon\in(0,\varepsilon_0]
\end{equation}
and
\begin{equation}\label{show2}
  d_{\mathbb{R}^n}(\psi,W_U\cap\mathbb{R}^n)=d_{\mathbb{R}^n}(\psi,U).
\end{equation}
To prove (\ref{show1}) let us define $W_U^\varepsilon\subset
C([0,T],\mathbb{R}^n)$ as
\[
  W_U^\varepsilon=\left\{\widehat {x}\in C([0,T],\mathbb{R}^n):
  x^{-1}_\varepsilon(t,\widehat {x}(t))\in U,\ {\rm for\ any\
  }t\in[0,T]\right\}.
\]
We claim that there exists
$\widehat{\varepsilon}_0\in(0,\varepsilon_0]$ such that
\begin{equation}\label{stepI}
  Q_\varepsilon x\not=x\quad{\rm for\ any\
  }x\in  \left(W_U\backslash W_U^\varepsilon\right)\cup
  \left(W_U^\varepsilon \backslash W_U \right) {\ \rm and\ any\
  }\varepsilon\in(0,\widehat{\varepsilon}_0].
\end{equation}
Assume the contrary, thus there exist sequences
$\{\varepsilon_k\}_{k\in\mathbb{N}}\subset(0,\varepsilon_0],$
$\varepsilon_k\to 0$ as $k\to\infty,$
$\{x_k\}_{k\in\mathbb{N}}\subset C([0,T],\mathbb{R}^n),$ such that
\begin{equation}\label{propI}x_k\in\left(W_U\backslash
W_U^{\varepsilon_n}\right)\cup\left(W_U^{\varepsilon_n} \backslash
W_U \right),\end{equation}
and
\begin{equation}\label{qwq}
x_k\to x_0{\rm\ as\ } k\to\infty{\rm\ \ where\ \
}Q_{\varepsilon_k}x_k=x_k.
\end{equation}
It is easy to see that (\ref{propI}) implies $x_0\in\partial W_U.$
This fact together with (\ref{qwq}) and the assumption that
$f_{x_0}(0)\not=0$ leads to a contradiction with the Malkin's
result (\ref{maln})-(\ref{mallemma}). Therefore, we have proved
that (\ref{stepI}) holds and thus
\[
  d(I-Q_\varepsilon,W_U)=d(I-Q_\varepsilon,W_U^\varepsilon)\quad{\rm
  for\ any\ }\varepsilon\in(0,\widehat{\varepsilon}_0].
\]
Since for any $\varepsilon\ge 0$ the sets $U$ and
$W_U^\varepsilon$ have a common core with respect to the
$T$-periodic problem for system (\ref{ps}), see (\cite{krazab}, \S
28.5), then, by (\cite{krazab}, Theorem 28.5), we have
\[
  d(I-Q_\varepsilon,W_U^\varepsilon)=d_{\mathbb{R}^n}
  (I-\Omega_\varepsilon,U)\quad{\rm
  for\ any\ }\varepsilon\ge 0
\]
and so (\ref{show1}) is proved. Finally, the proof of
(\ref{show2}) is obtained by means of the Leray-Schauder
continuation principle. In fact, let
\[
  U_\lambda=\left\{\xi\in\mathbb{R}^n:x^{-1}(\lambda t,\xi)\in U{\rm\ for\
  any\ }t\in[0,T]\right\},\quad\lambda\in[0,1],
\]
we now show that
\begin{equation}\label{impo}
  0\not\in\psi(\partial U_\lambda)\quad{\rm for\ any\
  }\lambda\in[0,1].
\end{equation}
Assume the contrary, thus there exists $\lambda_0\in[0,1]$ such
that $\xi_0\in\partial U_{\lambda_0}$ and $\psi(\xi_0)=0.$
Observe, that $x^{-1}(\lambda_0 t,\xi_0)\in\overline{U}$ for any
$t\in[0,T].$ Therefore, we have that there exists $t_0\in[0,T]$
such that $x^{-1}(\lambda_0 t_0,\xi_0)\in\partial U$ and from the
fact that $\psi(\xi_0)=0$ we have that $x^{-1}(\lambda_0 t,\xi_0)$
is constant with respect to $t\in[0,T].$ Hence we have
$x^{-1}(\lambda_0 t_0,\xi_0)=x^{-1}(0,\xi_0)=\xi_0$ and we obtain
that $\xi_0\in\partial U$ contradicting the fact that $\partial U$
contains only initial conditions of nondegenerate limit cycles of
(\ref{np}). By using the Leray-Schauder continuation principle
\cite{LS} (see also \cite{bro}, Theorem 10.7) from (\ref{impo}) we
now conclude that
\[
  d_{\mathbb{R}^n}(\psi,U_0)=d_{\mathbb{R}^n}(\psi,U_1).
\]
On the other hand $U_0=U$ and $U_1=W_U\cap\mathbb{R}^n$ and so the
proof of (\ref{show2}) is also complete.
\end{proof}

\begin{remark} From (\ref{kmnformU}) it follows that if
the limit cycle $x\in\mathfrak{S}^U$ touches $\partial U$ but it
does not intersect $\partial U$ then this cycle does not have any
influence in the evaluation of
$d_{\mathbb{R}^n}(I-\Omega_\varepsilon,W_U)$ with $\varepsilon>0$
sufficiently small.
\end{remark}

\section {Existence of $T$-periodic solutions}

By means of different choices of the set $W\subset
C([0,T],\mathbb{R}^n)$ we formulate in what follows some existence
results for $T$-periodic solutions to (\ref{ps}) in $W.$

\begin{theorem}\label{thm3} Assume that all the nonconstant $T$-periodic
solutions of (\ref{np}) are nondegenerate limit cycles of
(\ref{np}). Then for any open bounded set $W\subset
C([0,T],\mathbb{R}^n)$ containing all the constant solutions of
(\ref{np}) and satisfying the conditions
\[
\mathfrak{S}_W{\it\ is\ finite},\quad  f_x(0)\not=0\quad{\it for\
any\ }x\in \mathfrak{S}_W
\]
and
\[
(-1)^n
d_{\mathbb{R}^n}(\psi,W\cap\mathbb{R}^n)-\sum\limits_{x\in\mathfrak{S}_W:
~\Theta_W(x)\not=\emptyset}(-1)^{\beta(x)}
d_{\mathbb{R}}\left(f_x,\left(0,\min\{\Theta_W(x)\}\right)\right)
\not=0
\]
there exists $\varepsilon_0>0$ such that for any $\varepsilon\in
(0,\varepsilon_0]$ system (\ref{ps}) has a $T$-periodic solution
belonging to $W.$
\end{theorem}

\vskip0.2truecm The assumptions of Theorem \ref{thm3} implies that
the set $\mathfrak{S}_W$ contains only nondegenerate cycles of
(\ref{np}). Therefore, Theorem \ref{thm3} follows from Theorem
\ref{thm2} and the solution property of the Leray-Schauder
topological degree, see (\cite{krazab}, Theorem 20.5). Observe
that Theorem \ref{thm3} is an extension of (\cite{maw}, Corollary
4).

\vskip0.2truecm The next result provides conditions under which
the conclusion of (\cite{maw}, Theorem 2) remains valid also in
the case when $\partial W$ contains $T$-periodic solutions to
(\ref{np}).

\begin{corollary}\label{cor3} Assume that all the nonconstant
$T$-periodic solutions of (\ref{np}) are nondegenerate limit
cycles of (\ref{np}). Assume that there exists an open bounded set
$W\subset C([0,T],\mathbb{R}^n)$ containing all the constant
solutions of (\ref{np}) and satisfying the conditions
\begin{equation}\label{one} \mathfrak{S}_W{\it\ is\ finite}, \quad f_x(0)\cdot
f_x(\min\{\Theta_W(x)\})>0\quad{\it for\ any\ }x\in
\mathfrak{S}_W{\it \ with\ }\Theta_W(x)\not=\emptyset
\end{equation}
and
\[
d_{\mathbb{R}^n}(\psi,W\cap\mathbb{R}^n)\not=0.
\]
Then for any $\varepsilon>0$ sufficiently small system (\ref{ps})
has a $T$-periodic solution belonging to $W.$
\end{corollary}

\vskip0.2truecm The proof of the Corollary \ref{cor3} follows
directly from the fact that (\ref{one}) implies that
\[d_{\mathbb{R}}\left(f_x,\left(0,\min\{\Theta_W(x)\}\right)\right)=
0\quad{\rm for\ any\ }x\in \mathfrak{S}_W{\rm \ with\
}\Theta_W(x)\not=\emptyset\] (see \cite{krazab}, \S 3.2).

\vskip0.2truecm In what follows we give some applications of
Theorem \ref{thm1} to the problem of the existence of $T$-periodic
solutions to (\ref{ps}) near a nondegenerate limit cycle of
(\ref{np}). In the sequel $\rho(\xi,A)$ denotes the distance
between $\xi\in\mathbb{R}^n$ and $A\subset\mathbb{R}^n$ given by
$\rho(\xi,A)={\rm inf}_{\zeta\in A}\|\xi-\zeta\|.$ First, we state
the following result.

\begin{theorem}\label{thm4} Let $x_0$ be a nondegenerate $T$-periodic
limit cycle of (\ref{np}). Let $0\le\theta_1<\theta_2\le
\theta_1+\frac{T}{p}$ where $p\in\mathbb{N}$ and $\frac{T}{p}$ is
the least period of $x_0.$ Assume that
\begin{equation}\label{icond}
f_{x_0}(\theta_1)\cdot f_{x_0}(\theta_2)<0.
\end{equation}
Let $\Theta$ be the set of all zeros of $f_{x_0}$ on
$(\theta_1,\theta_2).$  Then, for any $\varepsilon>0$ sufficiently
small, system (\ref{ps}) has a $T$-periodic solution
$x_\varepsilon$ such that for any $t\in [0,T]$ we have
\begin{equation}\label{conve}
\rho\left(x_\varepsilon(t),x_0(t+\Theta)\right)\to 0\quad {\rm as\
  }\varepsilon\to 0.
\end{equation}
\end{theorem}

\begin{proof} Observe, that condition (\ref{icond}) implies
that
\[
f_{x_0}(\theta_1)\not=0\quad{\rm and}\quad f_{x_0}(\theta_2)\not=0
\]
and so the assumptions of Theorem \ref{thm1} are satisfied. Let us
fix $\alpha>0,$ from Theorem \ref{thm1} we have that there exists
$\delta_0>0$ such that for any
$\varepsilon\in(0,\delta_0^{1+\alpha})$ the topological degree
$d(I-Q_\varepsilon,W_{V_{\delta(\varepsilon)}})$ is defined with $
\delta(\varepsilon)=\varepsilon^{1/(1+\alpha)} $ and
\[
d(I-Q_\varepsilon,W_{V_{\delta(\varepsilon)}})=\,-\,
(-1)^{\beta(x_0)} d_{\mathbb{R}}(f_{x_0},(\theta_1,\theta_2)).
\]
From (\ref{icond}) we also have, see (\cite{krazab}, \S 3.2), that
$|d_{\mathbb{R}}(f_{x_0},(\theta_1,\theta_2))|=1$ and so for any
$\varepsilon\in(0,\delta_0^{1+\alpha})$ system (\ref{ps}) has a
$T$-periodic solution $x_\varepsilon$ such that
$x_\varepsilon(0)\in V_{\delta(\varepsilon)}.$ Moreover, from
property 1) of Theorem \ref{thm1} we have that
\begin{equation}\label{ss1}
\rho\left(x_\varepsilon(0),x_0([\theta_1,\theta_2])\right)\le
\delta(\varepsilon)=\varepsilon^{1/(1+\alpha)}.
\end{equation}
Let $u_\varepsilon(t)=x^{-1}(t,x_\varepsilon(t)),$ then, see e.g.
(\cite{non}, (13)-(19)),
\[
\dot u_\varepsilon(t)=\varepsilon \left(x'_{(2)}(
t,u_\varepsilon(t))\right)^{-1}\phi(t,x(t,u_\varepsilon(t),\varepsilon)).
\]
Therefore there exists $M_1>0$ such that
\begin{equation}\label{ss2}
  \|u_\varepsilon(0)-u_\varepsilon(t)\|\le M_1\varepsilon\quad {\rm for\ any\
  }\varepsilon\in(0,\delta_0^{1+\alpha}){\rm\ and\ any\ }t\in[0,T].
\end{equation}
On the other hand, $u_\varepsilon(0)=x_\varepsilon(0)$ and so from
(\ref{ss1}) and (\ref{ss2}) for any
  $\varepsilon\in(0,\delta_0^{1+\alpha})$  and any $t\in[0,T]$ we have that
\begin{equation}\label{uu1}
  \rho(u_\varepsilon(t),x_0([\theta_1,\theta_2]))\le
  \|u_\varepsilon(t)-x_\varepsilon(0)\|+
  \rho(x_\varepsilon(0),x_0([\theta_1,\theta_2]))\le
  \varepsilon^{1/(1+\alpha)}\left(1+M_1
  \varepsilon^{\alpha/(1+\alpha)}\right).
\end{equation}
Since for any $\theta\in[\theta_1,\theta_2]$ we have that
$\|x_\varepsilon(t)-x_0(t+\theta)\|=\|x(t,u_\varepsilon(t))-
x(t,x_0(\theta))\|$ and since, as it was already observed, in the
proof of Theorem \ref{thm1}, the function $x(\cdot,\cdot)$ is
continuously differentiable with respect to both variables we have
that there exists $M_2>0$ such that
\begin{equation}\label{uu2}
  \|x_\varepsilon(t)-x_0(t+\theta)\|\le M_2\|u_\varepsilon(t)-x_0(\theta)\|
  \quad {\rm for\ any\
  }\varepsilon\in(0,\delta_0^{1+\alpha}),\ t\in[0,T],\
  \theta\in[\theta_1,\theta_2].
\end{equation}
Substituting (\ref{uu1}) into (\ref{uu2}) we obtain that
\begin{equation}\label{fin0}
\rho(x_\varepsilon(t),x_0(t+[\theta_1,\theta_2]))\le
  \varepsilon^{1/(1+\alpha)}M_2\left(1+M_1
  \varepsilon^{\alpha/(1+\alpha)}\right)\quad {\rm for\ any\
  }\varepsilon\in(0,\delta_0^{1+\alpha}),\ t\in[0,T].
\end{equation}
Assume now that (\ref{conve}) is not true, thus there exist
$\delta_*>0$ and sequences
$\{\varepsilon_k\}_{k\in\mathbb{N}}\subset(0,\delta_0^{1+\alpha}),$
$\varepsilon_k\to 0$ as $k\to\infty,$ and
$\{t_k\}_{k\in\mathbb{N}}\subset[0,T]$ such that
\begin{equation}\label{popo}
  x_{\varepsilon_k}(t_k)\not\in B_{\delta_*}(x_0(t_k+\Theta))\quad{\rm
  for\ any\ }k\in\mathbb{N}.
\end{equation}
Without loss of generality we may assume that
$\{x_k\}_{k\in\mathbb{N}}$ and $\{t_k\}_{k\in\mathbb{N}}$ are
converging. From (\ref{fin0}) we have that there exists
$\theta_*\in[\theta_1,\theta_2]$ such that
\begin{equation}\label{po}
  x_k(t)\to x_0(t+\theta_*)\quad{\rm as\ }k\to\infty
\end{equation}
uniformly with respect to $t\in[0,T].$ By using \cite{mal} or
(\cite{malb}, Theorem p.~387) we can conclude from (\ref{po}) that
$f_{x_0}(\theta_*)=0.$ On the other hand, from (\ref{popo}) we
have that $x_0(t_0+\theta_*)\not\in
B_{\delta_*/2}(x_0(t_0+\Theta)),$ where $t_0=\lim_{k\to\infty}
t_k,$ and thus $x_0(\theta_*)\not\in B_{\delta_*/2}(x_0(\Theta)).$
This contradiction proves (\ref{conve}) and thus the proof is
complete.
\end{proof}

\noindent A topological degree approach to prove the existence of
periodic solutions to some classes of autonomous perturbed systems
can be found in \cite{bob} and \cite{rach}.

\begin{remark}\label{remark2} From the proof of
Theorem \ref{thm4} we see that Theorem \ref{thm1} also provides
information about the rate of the convergence of $T$-periodic
solutions of (\ref{ps}) to a limit cycle of (\ref{np}). In fact,
from (\ref{fin0}) we have that the distance between the graph of
the $T$-periodic solution $x_\varepsilon$ and the limit cycle
$x_0$ is of order $\varepsilon^{1/(1+\alpha)},$ where $\alpha>0$
is any positive constant.
\end{remark}

We are now in a position to establish some new existence results
of $T$-periodic solutions to (\ref{ps}). First, by using Theorem
\ref{thm4} we state in Corollary  \ref{cor4} a generalization of
the following Malkin's theorem, see (\cite{mal} and (\cite{malb},
Theorems pp.~387 and 392), (the same result with a more rigorous
proof is also given in (\cite{loud}, Theorem~1)). In fact, in
Corollary \ref{cor4} the Malkin's regularity assumptions are
weakened to conditions (\ref{reg}) and moreover
$(f_{x_0})'(\theta_0)$ can be $0.$

\vskip0.2truecm \ {\bf Malkin's theorem} {\it Let $\psi\in C^3,$
$\phi\in C^2.$ Let $x_0$ be a nondegenerate $T$-periodic limit
cycle of (\ref{np}). Assume that there exists $\theta_0\in[0,T]$
such that $f_{x_0}(\theta_0)=0$ and
\begin{equation}\label{mmm}
(f_{x_0})'(\theta_0)\not=0.
\end{equation}
Then for all $\varepsilon>0$ sufficiently small system (\ref{ps})
possesses a $T$-periodic solution $x_\varepsilon$ satisfying
\begin{equation}\label{conv}
  x_\varepsilon(t)\to x_0(t+\theta_0)\quad{\rm as\ }\varepsilon\to
  0, \quad t\in [0,T].
\end{equation}}

\begin{corollary}\label{cor4} Assume that $\psi$ and $\phi$ satisfy
conditions (\ref{reg}). Let $x_0$ be a nondegenerate $T$-periodic
limit cycle of (\ref{np}). Assume that there exists
$\theta_0\in[0,T]$ such that
\begin{equation}\label{monot}
f_{x_0}(\theta_0)=0{\quad \it and\quad}f_{x_0}{\it\ is\ strictly\
monotone\ at\ }\theta_0.
\end{equation}
Then for all $\varepsilon>0$ sufficiently small system (\ref{ps})
possesses a $T$-periodic solution $x_\varepsilon$ satisfying
(\ref{conv}).
\end{corollary}

The proof of Corollary  \ref{cor4} is a direct consequence of
Theorem \ref{thm4} with $\theta_1<\theta_0<\theta_2$ sufficiently
close to $\theta_0.$ We would like to observe that, under the
regularity assumptions of the Malkin's theorem, the asymptotic
stability of the resulting $T$-periodic solutions  can be also
established by means of the derivatives of the involved functions.
Clearly, under the weaker regularity assumptions (\ref{reg}) this
approach is impossible. On the other hand as shown in
\cite{ortega} some stability properties of the $T$-periodic
solutions to (\ref{ps}) can be derived from the value of the
degree $d(I-Q_\varepsilon,W_{V_\delta(\varepsilon)}),$ where
$V_\delta(\varepsilon)$  are the sets employed in the proof of
Theorem  \ref{thm4}.

The case when (\ref{mmm}) is not satisfied was treated by Loud in
\cite{loud}, we show here that, by using Theorem  \ref{thm4}, the
conditions of a related Loud's existence result can be
considerably simplified. Also for this case we do not provide here
any result about the stability of the resulting periodic solutions
as it has been done in \cite{loud}. In order to formulate the
Loud's existence result we introduce some preliminary notations.
First of all we need to translate and rotate the axes in such a
way that $x_0(0)=0$ and $\dot
x_0(0)=\left([x_0(0)]^1,0,...,0\right).$ Let
$x(\cdot,\xi,\varepsilon)$ be the solution of (\ref{ps})
satisfying $x(0,\xi,\varepsilon)=\xi.$ Let
$F(\xi,\varepsilon)=x(T,\xi,\varepsilon)-\xi,$ since the limit
cycle $x_0$ is nondegenerate then $n-1$ equations of the system
$F(\xi,\varepsilon)=0$ can be solved near $0$ with respect to some
$\xi^k,$ where $k\in\{ 1,2,...,n\}$ and as a result we obtain a
scalar equation $H(u,\varepsilon)=0.$ Let $D_{x_0}$ be the
discriminant of the equation
\[
\frac{1}{2}\frac{\partial^3 H}{\partial u^2\partial
\varepsilon}(0,0)m^2+ \frac{1}{2}\frac{\partial^3 H}{\partial
u\partial\varepsilon^2}(0,0)m+ \frac{1}{6}\frac{\partial^3
H}{\partial\varepsilon^3}(0,0)=0.
\]

We can now formulate the Loud's existence result, (\cite{loud},
Theorem 2).

\vskip0.2truecm  \ {\bf Loud's theorem. } {\it Let $\psi\in C^3,$
$\phi\in C^2.$ Let $x_0$ be a nondegenerate $T$-periodic limit
cycle of (\ref{np}). Assume that for some $\theta_0\in[0,T]$
satisfying $f_{x_0}(\theta_0)=0$ we have $(f_{x_0})'(\theta_0)=0.$
Finally, assume that
\begin{equation}\label{desc}
D_{x_0}>0 \quad {\it and} \quad (f_{x_0})''(\theta_0)=0.
\end{equation}
Then, for all $\varepsilon>0$ sufficiently small, system
(\ref{ps}) has a  $T$-periodic solution $x_{\varepsilon}$
satisfying (\ref{conv}).}

\vskip0.2truecm In the case when $f_{x_0}(\cdot)$ is identically
zero Loud in \cite{loud} has derived from the above theorem an
important result on the existence of $T$-periodic solutions to
(\ref{ps}) near $x_0.$ But even in the case when
$(f_{x_0})'''(\theta_0)\not=0$ to verify (\ref{desc}) is not a
feasible problem (here it is assumed $\phi\in C^3$). This is the
reason why it is of interest to state the following result which
is a particular case of Corollary \ref{cor4}.

\begin{corollary}\label{cor5} Let $\psi\in C^1,$  $\phi\in C^3.$ Let
$x_0$ be a nondegenerate $T$-periodic limit cycle of (\ref{np}).
Assume that for some $\theta_0\in[0,T]$ we have
\[
  f_{x_0}(\theta_0)=f_{x_0}'(\theta_0)=f_{x_0}''(\theta_0)=0,\quad
  f_{x_0}'''(\theta_0)\not=0.
\]
Then for all $\varepsilon>0$ sufficiently small system (\ref{ps})
has a  $T$-periodic solution $x_{\varepsilon}$ satisfying
(\ref{conv}).
\end{corollary}


\begin{thebibliography}{100}

\bibitem{bob}
N. A. Bobylev, M. A. Krasnosel'skii, A functionalization of the
parameter and a theorem of relatedness for autonomous systems,
Differencial'nye Uravnenija 6, 1946--1952 (1970). (In Russian).

\bibitem{bro}
R. F. Brown, A topological introduction to nonlinear analysis,
Birkh$\rm \ddot a$user Boston, Inc., Boston, MA, 1993.

\bibitem{maw}
A. Capietto, J. Mawhin, F. Zanolin, Continuation theorems for
periodic perturbations of autonomous systems, Trans. Amer. Math.
Soc. 329, 41-72 (1992).

\bibitem{dem}
B. P. Demidowicz, The mathematical theory of stability,
Wydawnictwa Naukowo-Techniczne, Warsaw, 1972.

\bibitem{non}
M. I. Kamenskii, O. Yu. Makarenkov, P. Nistri, Small parameter
perturbations of nonlinear periodic systems, Nonlinearity 17,
193--205 (2004).

\bibitem{dan}
M. I. Kamenskii, O. Yu. Makarenkov  and P. Nistri,  A new approach
in the theory of ordinary differential equations with small
parameter, Dokl. Math. Sci. 67, 36--38 (2003).

\bibitem{rach}
A. M. Krasnosel'skii, R. Mennicken and D. I. Rachinskii, Small
periodic solutions generated by sublinear terms, J. Differential
Equations 179, 97--132 (2002).

\bibitem{krazab}
M. A. Krasnosel'skii and P. P. Zabreiko, Geometrical methods of
nonlinear analysis. Fundamental Principles of Mathematical
Sciences 263. Springer-Verlag, Berlin, 1984.

\bibitem{kraop}
M. A. Krasnosel'skii, The operator of translation along the
trajectories of differential equations, Translations of
Mathematical Monographs, 19. Translated from the Russian by
Scripta Technica, American Mathematical Society, Providence, R.I.
1968.

\bibitem{LS} J.~Leray and J.~Schauder, Topologie et \'equations
fonctionnelles. Ann. Sci. \'Ecole Norm. Sup. (3) 51, 45--78
(1934).

\bibitem{loud}
 W. S. Loud, Periodic solutions of a perturbed autonomous system, Ann. of
Math. 70, 490--529 (1959).

\bibitem{malb}
I. G. Malkin, Some problems of the theory of nonlinear
oscillations, Gosudarstv. Izdat. Tehn.-Teor. Lit., Moscow, 1956.
(In Russian).

\bibitem{mal}
I. G. Malkin, On Poincar\'e's theory of periodic solutions,  Akad.
Nauk SSSR. Prikl. Mat. Meh. 13, 633--646 (1949). (In Russian).

\bibitem{mawPhD1}
J. Mawhin, Le Probl\`eme des Solutions P\'eriodiques en
M\'ecanique non Lin\'eaire, Th\`ese de doctorat en sciences,
Universit\'e de Li\`ege, 1969.

\bibitem{mawPhD2}
J. Mawhin, Degr\'e topologique et solutions p\'eriodiques des
syst\`emes diff\'erentiels non lin\'eaires, Bull. Soc. Roy. Sci.
Li\`ege 38, 308-398 (1969).

\bibitem{mawbvp}
J. Mawhin, Topological degree methods in nonlinear boundary value
problems, CBMS Regional Conf. Ser. Math. 40, Amer. Math. Soc.,
Providence R.I., 1979.

\bibitem{ortega}
R. Ortega, Some applications of the topological degree to
stability theory, in "Topological methods in differential
equations and inclusions" (Montreal, PQ, 1994), 377--409, NATO
Adv. Sci. Inst. Ser. C, Math. Phys. Sci., 472, Kluwer Acad. Publ.,
Dordrecht, 1995.

\bibitem{perron}
O. Perron, Die Ordnungszahlen der Differentialgleichungssysteme,
Math. Zeitschr. 31, 748-766 (1930).

\bibitem{rud}
W. Rudin, Principles of mathematical analysis. Second edition
McGraw-Hill Book Co., New York 1964.

\bibitem{sch}
K. R. Schneider, Vibrational control of singularly perturbed
systems, in Lecture Notes in Control and Information Science 259,
Springer Verlag, London, 397-408 (2001).


\end{thebibliography}
\end{document}